\documentclass[11pt]{amsart}
\usepackage{amssymb,amsfonts,amsthm,amsmath,enumerate,amscd}
\usepackage[english]{babel}
\newtheorem{theorem}{Theorem}[section]

\newtheorem{lemma}[theorem]{Lemma}
\newtheorem{corollary}[theorem]{Corollary}
\newtheorem{claim}[theorem]{Claim}
\newtheorem{proposition}[theorem]{Proposition}
\newtheorem{definition}[theorem]{Definition}
\newtheorem{example}[theorem]{Example}
\newtheorem{remark}[theorem]{Remark}

\newtheorem{convention}{Convention}

\newcommand{\field}[1]{\mathbb{#1}}
\newcommand{\C}{\field{C}}
\newcommand{\K}{\field{K}}
\newcommand{\N}{\field{N}}

\newcommand{\R}{\field{R}}

\newcommand{\wh}[1]{\widehat{#1}}

\newcommand{\aph}{{\alpha}}
\newcommand{\ba}{{\bf a}}
\newcommand{\bbo}{{\bf 0}}
\newcommand{\bB}{{\bf B}}
\newcommand{\be}{{\bf e}}
\newcommand{\bG}{{\bf G}}
\newcommand{\bH}{{\bf H}}
\newcommand{\bp}{{\bf p}}
\newcommand{\bP}{{\bf P}}

\newcommand{\bs}{{\bf s}}
\newcommand{\bS}{{\bf S}}
\newcommand{\bu}{{\bf u}}
\newcommand{\bv}{{\bf v}}
\newcommand{\bw}{{\bf w}}
\newcommand{\bx}{{\bf x}}
\newcommand{\by}{{\bf y}}
\newcommand{\bz}{{\bf z}}
\newcommand{\cA}{\mathcal{A}}
\newcommand{\cB}{\mathcal{B}}
\newcommand{\cC}{\mathcal{C}}
\newcommand{\Cn}{{\C^n}}

\newcommand{\cU}{\mathcal{U}}
\newcommand{\cV}{\mathcal{V}}
\newcommand{\cX}{\mathcal{X}}
\newcommand{\clos}{{\rm clos}}

\newcommand{\dd}{{\partial}}
\newcommand{\dist}{{\rm dist}}
\newcommand{\dlt}{{\delta}}

\newcommand{\gm}{{\gamma}}
\newcommand{\Gm}{{\Gamma}}
\newcommand{\gR}{{{\geq R}}}
\newcommand{\inn}{{\rm inn}}
\newcommand{\Kn}{{\K^n}}

\newcommand{\lbd}{\lambda}
\newcommand{\lgth}{{\rm length}}
\newcommand{\lr}{{{\leq r}}}
\newcommand{\lR}{{{\leq R}}}

\newcommand{\omg}{\omega}
\newcommand{\Omg}{\Omega}
\newcommand{\pro}{{{\rm proj}}}
\newcommand{\rd}{{\rm d}}

\newcommand{\Rn}{{\R^n}}
\newcommand{\Rp}{{\R^p}}
\newcommand{\sing}{{\rm sing}}

\newcommand{\ups}{{\upsilon}}
\newcommand{\ve}{\varepsilon}
\newcommand{\vp}{\varphi}

\newcommand{\zt}{{\zeta}}
\numberwithin{equation}{section}

\numberwithin{equation}{section}

\begin{document}
\title[Characterization of LNE curves]{Characterization of Lipschitz 
Normally Embedded complex curves}

\author[A. Costa]{Andr\'e Costa}
%
\author[V. Grandjean]{Vincent Grandjean}

\author[M. Michalska]{Maria Michalska}

\address{A. Costa, Centro de Ci\^encias e tecnologias, 	Universidade Estadual do Cear\'a, 
Campus do Itaperi, 60.714-903 Fortaleza - CE, Brasil}
\email{andrecosta.math@gmail.com}
\address{V. Grandjean, Departamento de Matem\'atica, Universidade Federal de Santa Catarina, 
Campus Universit\'ario Trindade CEP 88.040-900 Florian\'opolis-SC,
Brasil}
\email{vincent.grandjean@ufsc.br}
\address{M. Michalska, Wydzia\l{} Matematyki i Informatyki, Uniwersytet 
\L{}\'o{}dzki, Banacha 22, 90-238 \L{}\'o{}d\'z{}, Poland}

\email{maria.michalska@wmii.uni.lodz.pl}



\begin{abstract}
The main result of the paper states that a 
connected complex affine algebraic curve is Lipschitz normally 
embedded (shortened to LNE afterwards) in $\C^n$ if and only if its 
germ at any singular point is a finite union of non-singular
complex curve germs which are pairwise transverse, and its projective 
closure is in general position with the hyperplane at infinity.
To this aim, we completely characterize
complex analytic curves of a compact complex manifold which are LNE 
(regardless of the given Riemannian structure), and therefore we can relate 
when a projective algebraic curve is LNE in $\mathbb{CP}^n$ with the 
property of its general affine traces being LNE in $\Cn$.
This allows us to deduce that Lipschitz 
classification of LNE curves is topological. We describe a complete
invariant for the (outer and inner) Lipschitz equivalence of affine LNE 
curves and show that most such curves cannot be bi-Lipschitz homeomorphic
to a plane one.
\end{abstract}

\maketitle
\tableofcontents


\section{Introduction}
A Riemannian manifold $(M,g_M)$
carries a natural metric space structure $(M,d_M)$.
Any subset $S$ inherits the \em outer metric space 
structure $(S,d_S)$ \em 
where the distance $d_S$ between points of $S$ is
their distance in $(M,d_M)$. When $S$ is path-connected,
the infimum of the lengths of all the rectifiable curves in $S$ connecting 
two given points of $S$ is called the \em inner distance \em between these 
points, yielding the \em inner metric space structure $(S,d_\inn^S)$ \em 
over $S$. 	While investigating $C^\infty$ functions over closed subsets of 
the Euclidean space, Whitney observed in \cite{Whi1,Whi2} 
that  the metric spaces $(S,d_S)$ and $(S,d_\inn^S)$ need not be 
metrically equivalent, thus introduced the property: there exists a positive 
constant $L$ such that 
\begin{center}
	$
	(P) 
	\hfill \bs,\bs' \in S \; \Longrightarrow \; d_\inn^S(\bs,\bs') \; \leq 
	\; L \cdot d_S(\bs,\bs') , \hfill
	$
\end{center}
which guarantees that 
the metric spaces $(S,d_S)$ and 
$(S,d_\inn^S)$ are equivalent, since $d_S \leq d_\inn^S$. 

The equivalence of inner and outer distances
relates strongly with Lipschitz 
Stratifications, L-regular decompositions and Pancakes of sub-analytic sets
of Euclidean spaces \cite{Mos1,Par1,Kur1,Par2}. 
Property $(P)$ is called \em quasi-convexity \em in the fields of Lipschitz 
Analysis and Metric Geometry, while to Singularity Theorists, it
is known as being \em Lipschitz Normally Embedded \em 
(shortened below to LNE) since \cite{BiMo}. 

The last decade has shown growing interest in investigating analytic subset 
germs $(X,\bbo)$ admitting LNE representatives. 
Still, outside the case of complex curve germs and isolated complex surface 
singularity germs, 
the higher dimensional situation is a terra incognita yet to be 
explored:
the recent short survey \cite{FaPi} and the paper \cite{MeSa} present a 
rather exhaustive list of what is known on this topic. Similarly, asking
when an algebraic set is LNE nearby infinity has been addressed in a single 
paper, under very restrictive hypotheses \cite{FeSa}, which  recently gave 
rise to a generalization to necessary condition on tangent cones in 
\cite{DiRi}.
Moreover, very little 
is known on how to recognize which affine algebraic subsets of $\K^n$ are 
LNE, both when $\K=\C$ and $\K=\R$.
The only non-trivial example we are aware of is from  \cite{KePeRu}: it has 
non-isolated singularities, but is a $\K$-cone, thus essentially a
projective result. Let us mention that the new article \cite{Tar} describes 
fully 
the Lipschitz outer geometry of complex plane algebraic curves 
by a complete finite combinatorial invariant.  Last, \cite{FiMa} gives a 
complete 
description of quasi-convex subsets of the real plane, definable in some 
o-minimal structure. Still, neither of these two works characterizes LNE 
curves.

\bigskip
The goal of this paper is to start the investigation of the LNE property 
for subsets of the Euclidean space or of a given 
compact Riemannian manifold. We treat the case of complex curves in 
both settings: we fully characterize which complex algebraic curves of 
$\Cn$ are LNE and which complex analytic curves of a compact complex 
manifold are LNE (regardless of the Riemannian metric used). 

\medskip
The main result of this paper is the following:

\medskip\noindent
{\bf Theorem \ref{thm:main}.} \em Let $X$ be a complex algebraic curve of 
$\Cn$
of positive degree $d$. It is LNE if and only if the following conditions 
are satisfied:
(0) $X$ is connected; (i) At each singular point $\bx_0$ of $X$, the germ 
$(X,\bx_0)$ consists 
of finitely many pairwise transverse non-singular complex 
algebraic curve germs;
(ii) The intersection of the projective 
closure of $X$ 
in $\mathbb{CP}^n$ with the hyperplane at infinity consists exactly of 
$d$ points. 
\em

\medskip
Prior to establishing 
the affine result, and as a consequence of only local considerations, 
we obtain the following result in complex compact manifolds:

\medskip\noindent
{\bf Proposition \ref{prop:compact}.} \em Let $X$ be a complex analytic
curve of a complex compact manifold. It is LNE if and only if 
the following conditions are satisfied:
(0) $X$ is connected; (i) At each singular point $\bx_0$ of $X$, the
germ $(X,\bx_0)$
consists of finitely many pairwise transverse non-singular complex analytic 
curve germs.\em

\smallskip
Both above results require the same condition at singular points, but the 
fundamental difference between the two lies of course in what happens at 
infinity, situation whose treatment will occupy half of the paper. 
These results leads to the very close interplay between some
metric properties of a complex projective algebraic curve with those of its 
general affine traces.

\medskip\noindent
 {\bf Theorem \ref{thm:main-proj}.} 
\em
Let $X$ be an algebraic curve of $\mathbb{CP}^n$. The following statements 
are equivalent:
\\
(1) $X$ is LNE w.r.t. to the Fubini-Study metric.
\\
(2) There exists a hyperplane $H$ of $\mathbb{CP}^n$ intersecting $X$ 
only at non-singular points of $X$ at which $X$ and $H$ are transverse
and such that the affine trace $X\setminus H$ of $X$ in $\C_H^n := 
\mathbb{CP}^n\setminus H$ is LNE in $\C_H^n$.
\\
(3)
For any hyperplane $H$ of $\mathbb{CP}^n$ intersecting $X$ 
only at non-singular points of $X$ at which $X$ and $H$ are transverse, the 
affine trace $X\setminus H$ is LNE in $\C_H^n$.  
\em

\medskip
This latter result allows us to describe a complete invariant of (outer and inner) Lipschitz equivalence which simply encodes degree, number of singular 
points, genus and multiplicities at singular points. From this, we deduce the following result:

\medskip\noindent
{\bf Theorem \ref{thm:classify}.} 
\em
Let $X_1$ be an irreducible LNE complex algebraic curve of 
$\C^{n_1}$ and $X_2$ be an irreducible LNE complex algebraic curve of $\C^{n_2}$. Then 
$X_1$ and $X_2$ are bi-Lipschitz homeomorphic if and only if they are 
homeomorphic.
\em	

\medskip 
Lipschitz equivalence has been extensively studied for germs of analytic 
curves. In dimension $n_1=n_2=2$, two curves are outer Lipschitz 
equivalent if and only if there 
exists a homeomorphism germ $(\C^{n_1},\bbo) \to (\C^{n_1},\bbo)$ mapping 
$(X_1,\bbo)$ to $(X_2,\bbo)$, see \cite{PhTe}. Therefore, a complete 
outer Lipschitz equivalence invariant of plane curves was introduced 
recently in~\cite{Tar}. As shown in \cite{Fer}, the topological and outer Lipschitz classification are no longer the same for 
space curve germs, i.e. $n_1=n_2 \geq 3$. Furthermore, by the result
\cite{Tei},  a complex space curve germ is outer Lipschitz  equivalent to 
the germ of one of its generic linear plane projections. 
Corollary~\ref{cor:LNE-not-plane} establishes that this is no longer true for 
global curves.

\medskip
The paper is organized as follows. Section \ref{section:G-MP} presents
some necessary background such as the definitions of LNE and locally
LNE, as well as
Lemma \ref{lem:locally-LNE}, Corollary \ref{cor:submfd-LNE} and Proposition 
\ref{prop:lipschitz-graph}, which are among the essential bricks of the 
paper. Lemma \ref{lem:transverse-submfds} of Section \ref{section:ACFISG} 
gives a criterion for some isolated singularity set germ to be locally 
LNE. Corollary \ref{cor:LNE-representatives} 
about a one parameter family of LNE representatives
is of interest in its own. 
Both results are used in Section \ref{section:LC} to obtain 
Proposition \ref{prop:local-LNE}, a criterion for a complex curve germ to be 
LNE. It allows to deduce the aforementioned Proposition \ref{prop:compact},  
at the end of Section \ref{section:LC}. The next three sections develop the 
needed material to show the main 
result Theorem \ref{thm:main} in Section \ref{section:MR}: Sections 
\ref{section:PICn} and \ref{section:UPoAC} deal with the \em unbounded part
\em of the affine curve being LNE, while Section \ref{section:BPoAC} 
focuses on its \em bounded part. \em Section \ref{section:AvsP} 
explores the relationship between the LNE properties of an affine curve
of $\Cn$ and its projective closure.
In the last Section Theorem \ref{thm:classify} states that the complete 
invariant of the Lipschitz classification problem for irreducible 
complex LNE curves is actually a complete topological one for irreducible 
LNE curves. Last, by Corollary 
\ref{cor:LNE-not-plane}, the Lipschitz geometry of LNE 
curves cannot be reduced to that of the LNE plane curves. 
Since we aim at a  self-contained article, 
proofs of the already known Lemma \ref{lem:locally-LNE} and Proposition 
\ref{prop:local-LNE} are given.
%
%
%
%
%
%
%
%
%
%
%
%
%
%
%
%
%
%
%
%
%
%
%
%
%
%
%
%
%
%
%
%
%
%
%
%
%
%
%
%
%
%
%
%
%
%
%
\section{Geometric Preliminaries}\label{section:G-MP}
Let $\K$ be $\R$ or $\C$. The space $\Kn$ is equipped with the Euclidean 
distance denoted by $|-|$.
Let $eucl$ be the Euclidean/Hermitian metric tensor over $\Kn$.

Let $\bx$ be a point of $\K^n$. Let $r$ be a positive number.
The open ball of $\K^n$ centred at $\bx$ with radius $r$ is 
$B_r^{n_\K}(\bx)$, 
for $n_\K := \dim_\R \Kn$.
Its closure is $\bB_r^{n_\K}(\bx)$. The Euclidean sphere of $\K^n$ centred
at $\bx$ and of radius $r$ is $\bS_r^{n_\K-1}(\bx)$. When $\bx$ is the origin 
$\bbo$ of $\K^n$, we will simply write $B_r^{n_\K}$, $\bB_r^{n_\K}$ and 
$\bS_r^{n_\K-1}$.

\begin{convention} A Riemannian manifold $(M,g_M)$ consists of a $C^k$
or $\K$-analytic manifold $M$, for $k \in \N_{\geq 1}\cup \infty$, 
of positive $\K$-dimension $m$, and of a continuous Riemannian metric
$g_M$ on $M$.
\end{convention}

Let $(M,g_M)$ be a Riemannian manifold. Thus  $g_M$ induces the 
distance $d_M$ on $M$ defined as the infimum of the lengths of rectifiable 
curves connecting any given pair of points. 

Any subset $S$ of $M$ admits two natural metric space structures on $S$ 
inherited from $(M,g_M)$:
\begin{definition}\label{def:metrics}
1) The \em outer distance on $S$ \em is the distance function $d_S$ on 
$S\times S$ obtained by restricting $d_M$ to $S\times S$. 

\smallskip\noindent
2) The \em inner distance on $S$ \em is the function 
$$
d_\inn^S: S\times S \to [0,+\infty]
$$
such that given $\bs,\bs'\in S$, the number $d_\inn^S(\bs,\bs')$ is the infimum of the lengths
of the rectifiable paths lying in $S$ joining $\bs$ and $\bs'$. If there is no such a path, then $d_\inn^S(\bs,\bs') = \infty$.   
\end{definition}
Observe that $d_S \leq d_\inn^S$. We recall the following notion.
\begin{definition}\label{def:quasi-isometric}
Let $(X,d_i)$, where $i=1,2$, be two metric space structures over the
space $X$. The metric space structures $(X,d_1)$ and $(X,d_2)$ are \em 
equivalent \em if 
there exists a constant $L>1$ such that

\begin{center}
$
\displaystyle{
\bx,\bx'\in X \; \Longrightarrow \; \frac{1}{L}\cdot d_1(\bx,\bx') \; 
\leq \; d_2(\bx,\bx') \; \leq \; L \cdot d_1(\bx,\bx').
}
$
\end{center}
\end{definition}
An elementary question arising from presenting these two definitions is 
when are the metric spaces $(S,d_S)$ and $(S,d_\inn^S)$ equivalent?  
This motivates the next
\begin{definition}\label{def:LNE}
A subset $S$ of the Riemannian manifold $(M,g_M)$ is \em Lipschitz normally 
embedded \em (shortened to LNE) if the identity mapping $(S,d_S)  \to 
(S,d_\inn^S)$ is bi-Lipschitz: there exists a positive constant $L$ such 
that
$$
\bs,\bs' \; \in \; S \; \Longrightarrow \; d_\inn^S(\bs,\bs') \; \leq \;
L \cdot d_S(\bs,\bs').
$$
Any constant $L$ satisfying the previous inequality is called a \em LNE 
constant of $S$. \em 
\end{definition}
\noindent
We will simply use the expression LNE, but it must never be forgotten
that it is with respect to the metric structure given by the Riemannian 
metric $g_M$ on $M$. 
\begin{remark}\label{rem:compact-manifold}
(i) The metric space structure $(M,d_M)$ associated with the Riemannian 
manifold $(M,g_M)$ is LNE, with LNE constant $1$. Moreover each chart is 
locally bi-Lipschitz.
\\
(ii) Since any two continuous Riemannian metrics on a $C^1$ compact 
Riemannian manifold $M$ yield equivalent metric space structures,
the property of being LNE depends only on the $C^1$ differentiable structure. 
Moreover, following Sullivan's result 
\cite{Sul}, a compact topological manifold of dimension 
different from $4$ has
a unique metric space structure, up to equivalence.
Therefore the property of being LNE in a compact 
Riemannian manifold of dimension which is not 4 depends only on the ambient topology.
\end{remark}

\medskip
Finding a local characterisation of a property is usually a first step to 
understand it globally, but the property \em 
being LNE \em does not pass to germs: the open ball $B_r^2$ centred at $\bbo$ 
of positive radius $r$ is a 
LNE representative of the germ $(\R^2,\bbo)$, while $\{(x,y) : y^2 -(x-r)^3 
>0 \}\cap B_{2r}^2$ is a non-LNE one.

We introduce 
the following notion (see also \cite{KePeRu}),
that will be central to the paper.
\begin{definition}\label{def:loc-LNE}
1) A subset $S$ of the Riemannian manifold $(M,g_M)$ is \rm locally LNE at 
the point $\bs$ of $\clos(S)$, \em the closure of $S$ in $M$, if there 
exists an open neighbourhood $\cU$ of $\bs$ in $M$ such
that $S\cap \cU$ is LNE.

\noindent
2) The subset $S$ is \em locally LNE \em if it is locally LNE at each
point of $\clos(S)$.
\end{definition}
The next Lemma 
generalizes in our setting the Euclidean result \cite[Proposition 
2.4]{KePeRu}. It
presents a simple criterion for a compact subset of $(M,g_M)$ to be LNE. 
\begin{lemma}\label{lem:locally-LNE}
Let $S$ be a compact connected subset of the Riemannian manifold $(M,g_M)$. 
It is LNE if and only if it is locally LNE.
\end{lemma}
\begin{proof}
We follow the proof of \cite{KePeRu}. 
When $S$ is LNE, taking $\cU = M$ yields that $S$ is locally LNE.

\medskip
For each point $\bs$ of $S$, let $\cU_\bs$ be a neighbourhood of $\bs$ in
$M$ such that $U_\bs := S \cap \cU_\bs$ is LNE. Thus $S$ is covered by 
$U_1,\ldots,U_N,$ where $U_i := U_{\bs_i}$ for some points 
$\bs_1,\ldots,\bs_N$ of $S$.
Let $L_0$ be a LNE constant of $U_1,\ldots,U_N$. 
Consider the following function
$$
f: S\times S \setminus \Delta \to \R, \;\;
(\bs,\bs') \mapsto \frac{d_\inn^S(\bs,\bs')}{d_S(\bs,\bs')},
\;\; {\rm where} \;\;
\Delta := \{(\bs,\bs) \in S \times S\}.
$$
Note that $f$ is bounded below by $1$. Let $V := \cup_{i=1}^N U_i\times U_i$
which is an open neighbourhood of $\Delta$ in $S\times S$. Observe that it is 
connected.
Since $f$ is continuous and $S\times S \setminus V$ is compact it admits a 
maximum $L_1$. Let $(\bs,\bs')$ in $V\setminus \Delta$, thus there exists $i$ 
such that 
$(\bs,\bs') \in U_i\times U_i$. Therefore $f(\bs,\bs') \leq L_0$. 
Thus $S$ is LNE with LNE constant $\max(L_0,L_1)$.
\end{proof}
We recall that a subset $S$ of a $C^k$ manifold $M$, $k\geq 1$, is a \em 
$C^1$ sub-manifold (possibly with boundary) of $M$ \em if: i) $S$ is a $C^1$ 
manifold (possibly with boundary); ii) the manifold topology coincides 
with the subspace topology; iii) the inclusion mapping 
$S\hookrightarrow M$ is a $C^1$ injective immersion.
As a corollary of the previous criterion we obtain the following result that 
we will use extensively in the paper.
\begin{corollary}\label{cor:submfd-LNE}
Let $N$ be a $C^1$ compact connected sub-manifold of the Riemannian manifold 
$(M,g_M)$ possibly with $C^1$ boundary. Then $N$ is LNE.
\end{corollary}
\begin{proof}
Since $N$ is connected we can assume that $M$ is also connected. Let $m,n$
be the respective dimensions of $M,N$. By Lemma \ref{lem:locally-LNE}, it suffices to show that $N$ is locally LNE.

\smallskip\noindent
$\bullet$ \em Assume that $(M,g_M) = (\R^m,eucl)$ and $n=m$. \em 
For each $\bx$ in $N\setminus \dd N$, there exists a positive radius $r$ 
such that $\bB_r^m(\bx)$ is contained in $N$. Since the ball is convex, it is 
LNE, thus $N$ is locally LNE at $\bx$. Let $\bx\in\dd N$. There exists 
a positive radius $r$ and a $C^1$ diffeomorphism
$$
\phi_\bx^\dd: N\cap B_r^m(\bx)\to B_{2r}^{m-1} \times [0,2), 
$$
mapping $\dd N\cap B_r^m(\bx)$ onto $B_2^{m-1} \times 0$. 
Up to shrinking $r$ we can assume that $\phi_\bx^\dd$ is bi-Lipschitz: there 
exists a positive constant $A$ such that 
$$
\by,\by' \in N\cap B_r^m(\bx) \; \Longrightarrow \; |\by - \by'| \leq A\cdot
|\phi_\bx^\dd(\by) - \phi_\bx^\dd(\by')| \;\leq \; A^2\cdot|\by - \by'|. 
$$
Since $B_{2r}^{m-1} \times [0,2)$ is convex in $\R^m$ it is LNE, thus 
$N\cap B_r^m(\bx)$ is also LNE. Therefore, $N$ is locally LNE at each of its points.

\smallskip\noindent
$\bullet$ \em Assume that $(M,g_M) = (\R^m,eucl)$ and $n<m$. \em 
For each $\bx$ in $N$, let $T_\bx$ be the tangent space of $N$ at $\bx$ and 
let $\pi_\bx :\R^m \to T_\bx$ be the orthogonal projection onto $T_\bx$.
There exists a positive radius $r$ such that $p_\bx$, the restriction of 
$\pi_\bx$ to $N\cap B_r^m(\bx)$ is a $C^1$ diffeomorphism onto its image, and
further we find
$$
\by,\by' \in N\cap B_r^m(\bx) \; \Longrightarrow \; |\by - \by'| \geq 
|p_\bx(\by) - p_\bx(\by')| \;\geq \; \frac{1}{2}|\by - \by'|. 
$$
Let $Z:=\pi_\bx(N)$. Therefore $(Z,\bx)$ is a $C^1$ sub-manifold possibly with boundary
of $T_\bx$ and thus $Z$ is locally LNE at $\bx$ w.r.t. the Euclidean metric 
of $T_\bx$ by the same argument as in the previous case. Since $T_\bx$ is still an euclidean space 
for the restriction of the Euclidean metric of $\R^m$, there exists a 
positive radius $r'$ such that $Z\cap B_{r'}^m(\bx)$ is LNE since 
$T_\bx\cap B_{r'}^m(\bx')$ is an Euclidean ball of $T_\bx$ of radius $r'$. 
Since $p_\bx$ is bi-Lipschitz we deduce that $p_\bx^{-1}(Z\cap 
B_{r'}^m(\bx))$ is LNE whenever $2r' \leq r$. 

\smallskip
We conclude that any sub-manifold of $\R^m$ is locally LNE at each of its 
points.

\smallskip\noindent
$\bullet$ \em General Case. \em
For each $\bx$ in $N$, let $\Phi_\bx :\cV_\bx \to \R^m$ be a $C^k$ chart of 
$M$ mapping $\bx$ onto $\bbo$. Therefore the germ $(\Phi_\bx(N),\bbo)$ is 
that of $C^1$ sub-manifold, possibly with boundary, and is of dimension $n$. 
Let $\dlt_\bx$ be the distance on $B_1^m$ obtained from the pushed-forward 
Riemannian metric $h_\bx := (\Phi_\bx)_*g_M$. Let $V_\bx := \Phi_\bx^{-1}(B_1^m)$. 
Since $\Phi_\bx$ yields a bi-Lipschitz 
homeomorphism $(V_\bx,g_M|_{V_\bx}) \to (B_1^m,eucl|_{B_1^m})$, the metric
space structures $(B_1^m,\dlt_\bx)$ and $(B_1^m,|-|)$ are equivalent. 
Since $\Phi_\bx(N)$ is locally LNE at  
$\bbo$ in $\R^m$, it is also locally LNE at $\bbo$ w.r.t. $h_\bx$. Therefore 
$N$ is locally LNE at $\bx$. 
\end{proof}
In the Euclidean context, the following result complements
Corollary \ref{cor:submfd-LNE}. 
\begin{lemma}\label{lem:LNE-complement}
Let $N$ be a compact $C^1$ sub-manifold with $C^1$ boundary $\dd N$ of 
$\Rn$ of dimension $n$. Then each connected component of $\Rn\setminus N$ 
and of $\Rn \setminus (N \setminus \dd N)$ are LNE.
\end{lemma}
\begin{proof}
We will treat the case $\Rn \setminus (N \setminus \dd N)$. The remaining 
one will follow from the exact same demonstration. 
Let $E:= \Rn\setminus (N\setminus \dd N)$. Observe that $\dd E = \dd N$, and that $E$ has a single unbounded connected component. Since the bounded components of $E$ are compact manifolds with  boundary, they are LNE. 
Without loss of generality we can assume that $E$ is connected.

Let $B_1,\ldots,B_s,$ be the connected components of $\dd E$. Define the positive number
$$
\dlt : = \min_{1\leq i < j\leq s} \dist(B_i,B_j), 
$$
with the convention that $\dlt =\infty$ if $\dd E$ is connected. 

Let $\bx, \bx'$ be two distinct points of $E$, and let $I$ be the segment of 
$\Rn$ with end points $\bx,\bx'$. The segment $I$ is oriented from $\bx$ to $\bx'$, thus 
is totally ordered. When $I$ intersects with $\dd E$,  we can assume without loss of 
generality that $\bx,\bx'$ are both in $\dd E$. 
Since $\dlt$ is positive, there exists finitely many
open intervals $J^k = (\bx_k^1,\bx_{k+1}^0)$ of $I$, 
for $k= 1,\ldots,p$, with the following properties:
\begin{itemize}
\item[i)] 
$\bx = \bx_1^0 \leq \bx_1^1 < \bx_2^0\leq  \ldots < \bx_p^0\leq \bx_p^1 = 
\bx'$;
\item[ii)] For each $k=1,\ldots,p$, there exists $i_k \geq1$ such that $\bx_k^0,\bx_k^1 \in B_{i_k}$;
\item[iii)] For each $k=1,\ldots,p-1$, the interval $J^k$ is contained in $I\setminus N$;
\item[iv)] $i_k\neq i_{k+1}$ for each $k=1,\ldots,p-1$;
\item[v)] $p$ is maximal for the four previous conditions. 
\end{itemize}
Observe that each connected component of
the compact set $I\cap \dd E$ is contained in one of the intervals $[\bx_k^0,\bx_k^1]$,
where $k=1,\ldots,p$.

Since each connected component of $\dd E$ is LNE by Corollary 
\ref{cor:submfd-LNE}, let $A$ be a common LNE constant. For each 
$k=1,\ldots,p$, we find
$$
d_\inn^E (\bx_k^0,\bx_k^1) \leq d_\inn^{B_{i_k}} (\bx_k^0,\bx_k^1) \leq A\, |\bx_k^0 - \bx_k^1|.
$$ 
Since the segment $J^k$ is contained in $E$, we find  
$$
d_\inn^E (\bx_k^1,\bx_{k+1}^0) = |\bx_k^1 - \bx_{k+1}^0| \leq A\, 
|\bx_k^1 - \bx_{k+1}^0|.
$$ 
We deduce 
$$
d_\inn^E (\bx,\bx') \leq \sum_{k=1}^{p} d_\inn^E (\bx_k^0,\bx_k^1) + \sum_{k=1}^{p-1} d_\inn^E (\bx_k^1,\bx_{k+1}^0)\leq A |\bx - \bx'|,
$$
and the lemma is proved.
\end{proof}
\begin{remark}
The analogue of Lemma \ref{lem:LNE-complement} is also true when the ambient 
metric space is a connected Riemannian manifold $(M,g_M)$.
The proof of this more general case goes 
through similarly, using estimates of curves of length close enough to 
the outer/inner distance. 
Since we only use such a result in the Euclidean context, we 
presented the statement and proof in this simpler setting.
\end{remark}

\medskip
Let $(M,g_M)$ and $(P,g_P)$ be Riemannian manifolds. 
The manifold $M\times P$ is thus equipped
with the Riemannian product metric $g_{M\times P}:= g_M\otimes g_P$,
whose associated distance function is $d_M + d_P$.
Given a subset $S$ of $(M,g_M)$, we can speak of a Lipschitz mapping 
$F:(S,d_S) 
\to (P,d_P)$ with Lipschitz constant $K_F$:
$$
\bs,\bs' \;\in\; S \;\Longrightarrow \; 
d_P(F(\bs),F(\bs')) \; \leq \; K_F \cdot d_S(\bs,\bs') = K_F \cdot d_M 
(\bs,\bs').
$$
The space $S\times P$ comes naturally equipped with the product distance
$d_{S\times P} := d_S + d_P$. We conclude this section with the following very 
useful, though simple, result.
\begin{proposition}\label{prop:lipschitz-graph} 
Let $S$ be a LNE subset of $(M,g_M)$. If $F:(S,d_S) \to
(P,d_P)$ is a Lipschitz mapping, then its graph is LNE in $(M\times P,
d_{M \times P})$.
\end{proposition}
\begin{proof}
Since $F$ is Lipschitz there exists a positive constant $K_F$ such that
$$
\bs,\bs' \;\in\; S \;\Longrightarrow \; 
d_P(F(\bs),F(\bs')) \; \leq \; K_F \cdot d_S(\bs,\bs').
$$ 
Let $\gm:[0,1] \to S$ be a rectifiable curve and let $\gm_F:[0,1] \to S\times 
 P$ be the lift of $\gm$ to $G_F$, the graph of $F$. The result follows from
the following estimate
$$
{\rm length}_{M\times P}(\gm_F) \; \leq  (1+K_F) \cdot 
{\rm length}_M(\gm)
$$
where $\lgth_Q$ is the length taken w.r.t. to the Riemannian metric
$g_Q$, for $Q=M,\, M\times P$.
\end{proof}
\begin{remark}
When $M=\R^m,P=\Rp$, their product distance $d_M + d_P$
is not the Euclidean distance of $\R^{m+p}$, but it is equivalent to the 
Euclidean distance of $\R^{n+p}$, and this is how we will use Proposition 
\ref{prop:lipschitz-graph} in the complex affine context. 
\end{remark}
%
%
%
%
%
%
%
%
%
%
%
%
%
%
%
%
%
%
%
%
%
%
%
%
%
%
%
%
%
%
%
%
%
%
%
%
%
%
%
%
%
%
%
%
%
%
%
%
%
%
%
\section{A criterion for isolated singularity germs }\label{section:ACFISG}

Let $\bbo$ be the origin of the Euclidean space $\Rn$. If $S$ is any subset 
of $\Rn$ containing the origin, let $S^*$ be $S\setminus \bbo$.
\begin{lemma}\label{lem:transverse-submfds}
Let $(N_1,\bbo),\ldots,(N_s,\bbo),$ be germs at $\bbo$ of $C^1$ sub-manifolds 
of $\Rn$ of positive dimensions and codimensions. Assume that $(Y,\bbo) := 
\cup_{j=1}^s (N_j,\bbo)$ has an isolated singularity at $\bbo$. 
\\
\indent
The germ $(Y,\bbo)$ is locally LNE at $\bbo$ if and only if $T_\bbo N_j \cap 
T_\bbo N_k = \{\bbo\}$ for each $1\leq j<k\leq s$.
\end{lemma}
\begin{proof}
For each $j=1,\ldots,s,$ let $T_j := T_\bbo N_j$ and let $T_j^\perp$ be the 
orthogonal supplement of $T_j$. Since each $(N_j,\bbo)$ is a $C^1$ 
sub-manifold, there exists a positive radius $r_Y > 0$, and a $C^1$-mapping  
$$
G_j : T_j \cap B_{2r_Y}^n \to T_j^\perp  
$$
such that a representative of $(N_j,\bbo)$ is the graph $\Gamma_j$ of $G_j$.
Since $D_\bbo G_j = \bbo$ and $G$ is $C^1$, up  to shrinking $r_Y$ the mapping $G_j$ is 
Lipschitz over $\bB_r^n$ for any radius $0<r\leq r_Y$, with Lipschitz 
constant tending to $0$ as $r$ goes to $0$. For $j=1,\ldots,s,$ we define
$$
N_j^\lr := \Gm_j \cap \bB_r^n.
$$
Whenever $r_Y$ is small enough, we can further assume that $r_Y$ is such 
that $N_j^\lr$ is LNE for any $r\leq r_Y$, since they are connected compact
$C^1$ sub-manifolds with boundaries
$$
\dd N_j^\lr = \Gm_j\cap \bS_r^{n-1}.
$$
Each $\dd N_j^\lr$ is diffeomorphic to $\bS^{\dim T_j - 1}$.

Given a pair $1\leq j<k \leq s$, and $r\leq r_Y$, let $\bx_l \in 
(N_l^\lr)^*$ for $l=j,k$. Let $r_l > 0$ be defined as follows:
$$
r_l := |\bx_l|.
$$ 
Let $2\aph \in [0,\pi]$  be the (non-oriented) angle between $\bx_j$ and 
$\bx_k$. From plane geometry classes we recall that
$$
\left| \frac{\bx_j}{r_j} - \frac{\bx_k}{r_k} \right| = 2 \sin \aph 
$$
and further remember that
\begin{equation}\label{eq:xj-xk-equal}
|\bx_j - \bx_k|^2 = (r_j-r_k)^2\cos^2\aph + (r_j+r_k)^2 \sin^2\aph.
\end{equation}
By definition of $\aph$, we deduce that 
\begin{equation}\label{eq:xj-xk-inequal}
(r_j + r_k) \sin \aph   \; \leq \; |\bx_j - \bx_k| 
\; \leq \; r_j + r_k.
\end{equation}

Finally let $Y_\lr := Y\cap \bB_r^{2n}$,  $d_\inn^\lr := d_\inn^{Y_\lr}$
and $d_j^\lr := d_\inn^{N_j^\lr}$, for $1\leq j\leq s$.

\medskip\noindent
$\bullet$ \em Assume that $T_j \cap T_k = \bbo$, for each pair $j\neq k$. \em

For $j=1,\ldots,s,$ we define the link of $T_j$ at $\bbo$
$$
S_j := T_j \cap \bS_1^{n-1}. 
$$
By hypothesis, we get 
\begin{equation}\label{eq:dist-SjSk}
\dlt := \min_{1\leq j<k\leq w} \dist(S_j,S_k) > 0.
\end{equation}
We can require that $r_Y$ is 
small enough such that for each $j=1,\ldots,s,$ and each $\bx_j \in 
(N_j^{\leq r_Y})^*$ the followings estimates hold true
\begin{equation}\label{eq:Sj-xj-normed}
\dist\left(S_j,\frac{\bx_j}{r_j}\right) \; \leq \;  \frac{\dlt}{4}\cdot
\end{equation}
For each $1\leq j<k\leq s$, each $\bx_j \in N_j^*$, $\bx_k\in N_k^*$, 
we also find the following estimates
\begin{equation}\label{eq:xjxk-normed}
\left| \frac{\bx_j}{r_j} - \frac{\bx_k}{r_k} \right|\geq 
\frac{\dlt}{2}\cdot
\end{equation}
By choice of $r_Y$ and $j,k$ we have 
$$
\left| \frac{\bx_j}{r_j} - \frac{\bx_k}{r_k} \right| = 2\sin(\aph)\; \geq 
\; \frac{\dlt}{2}.
$$
From Estimate \eqref{eq:xj-xk-inequal}, we deduce 
\begin{equation}\label{eq:xj-xk-origin}
\frac{\dlt}{4} \cdot (r_k+r_j) \leq |\bx_j - \bx_k| .
\end{equation}
as well as
\begin{equation}\label{eq:xl-origin}
|\bx_l| = r_l \leq \; \frac{4}{\dlt} \cdot |\bx_j - \bx_k|
\;\; {\rm for} \;\; l=j,k.
\end{equation}
Since $Y$ has an isolated singularity at $\bbo$, we obviously have 
$$
d_\inn^\lr (\bx_j,\bx_k) \leq d_j^\lr(\bx_j,\bbo) + d_k^\lr(\bbo,\bx_k), 
$$ 
combining Estimate \eqref{eq:xl-origin} with 
$N_j^\lr$ and $N_k^\lr$ being LNE, each with LNE constant $A$, yield an expected 
inequality
$$
d_\inn^\lr (\bx_j,\bx_k) \; \leq \; \frac{8A}{\dlt}\cdot|\bx_j -\bx_k|.
$$
Since each  $N_j^\lr$ is LNE, we obtain that $Y^\lr$ is LNE.

\medskip\noindent
$\bullet$ \em Assume that $(Y,\bbo)$ is locally LNE at $\bbo$. \em

Since $(Y,\bbo)$ has an isolated singularity at $\bbo$, we can assume
that $r_Y$ is small enough so that the following identity 
holds true
$$
j<k \; \Longrightarrow \; (N_j^{\leq r_Y})^* \cap (N_k^{\leq 
r_Y})^*
= \emptyset.
$$
Given any pair $1\leq j<k \leq s$, observe that to connect $\bx_j$ to $\bx_k$ 
within $Y^{\leq r_Y}$ it is necessary to go through $\bbo$. Therefore 
\begin{equation}\label{eq:too-long}
d_\inn^Y(\bx_j,\bx_k) \geq r_j + r_k
\end{equation}
Let $E := \{\{i_1,\ldots,i_t\}$ such that $1\leq i_1< \ldots < i_t\leq 
s,$ with $t\geq 2$ and $\dim \cap_{j=1}^t T_{i_j} \geq 1\}$. 

Assume that $E$ is not empty. Let $J$ be an element of $E$ and let 
$$
P_ J:= \cap_{j\in J} T_j.
$$
For any $\bp \in P_J$ and each $j \in J$, let $\bx_j$ be the point 
$(\bp,G_j(\bp))$ of $N_j$. Since $P_J$ is tangent at $\bbo$ to each $N_l$
and $D_\bbo G_l$ is null, for any $l \in  J$, as $\bp$ goes to $\bbo$ in 
$P_J$ we find that
$$
|\bx_j|,|\bx_k| = |\bp| + o(|\bp|) \;\; {\rm and} \;\;
|\bx_j - \bx_k| = o(|\bp|) 
$$
for any pair $j,k$ of $J$ with $j\neq k$. Comparing the last estimate 
with Estimate \eqref{eq:too-long} shows that the germ $(Y,\bbo)$ cannot 
admit a LNE representative in any neighbourhood of $\bbo$. Necessarily $E$ is empty.
\end{proof}
A careful reading of the proof of the only if part of Lemma 
\ref{lem:transverse-submfds} yields the following result about LNE 
representatives
\begin{corollary}\label{cor:LNE-representatives}
Let $Y$ be a representative of the germ $(Y,\bbo)$ of Lemma 
\ref{lem:transverse-submfds}. If $(Y,\bbo)$ is LNE, there exists a 
positive radius $r_Y$ such that for each radius $r\in (0,r_Y]$, the 
subsets $Y\cap \bB_r^n$ and $Y\cap B_r^n$ are LNE.
\end{corollary}
%
%
%
%
%
%
%
%
%
%
%
%
%
%
%
%
%
%
%
%
%
%
%
%
%
%
%
%
%
%
%
%
%
%
%
%
%
%
%
%
%
%
%
%
%
%
%
%
%
%
%
%
%
\section{The local case}\label{section:LC}

We work in the local complex analytic category. 

We present here Proposition \ref{prop:local-LNE}, a known local criterion
for a complex analytic curve germ to be LNE (see for instance 
\cite{NePi,DeTi}), with a complete proof. 
This result is used to obtain Proposition \ref{prop:compact} which fully 
characterizes complex curves which are LNE in a given compact complex 
manifold (independently of any prescribed continuous Riemannian metric).

\begin{convention} A curve germ is a complex analytic curve germ
at a prescribed point.
\end{convention}

\medskip
Let $(Y,\bbo)$ be a curve germ of $\Cn$ at $\bbo\in\Cn$. 
Let $m(Y,\bbo)$ be the multiplicity of $Y$ at $\bbo$.
We recall  that $Y$ is non-singular at $\by_0$ if and only if 
$m(Y,\by_0) = 1$.

Assume that $Y$ is irreducible at $\bbo$ and write $m$ for $m(Y,\bbo)$. 
Let $L$ be the tangent cone of $Y$ at $\bbo$, which is a 
complex line through $\bbo$. The restriction to $(Y,\bbo)$ of the 
(orthogonal) projection $\Cn \to L$ is a complex analytic finite
mapping $p_L:(Y,\bbo) \to (L,\bbo)$ inducing a holomorphic $m$-sheeted
covering $(Y^*,\bbo) \to (L^*,\bbo)$.
After an orthonormal change of complex coordinates we can assume 
that 
$$
L := \{\bx = (x_1,\ldots,x_n) \in \Cn \, : \, x_2 =\ldots = x_n = 0\}
= \C \times \bbo.
$$
From Puiseux Theorem, there exists a holomorphic map germ  
$F:(\C,0) \to (\C^{n-1}, \bbo)$ such that
\begin{equation}\label{eq:puiseux}
(Y,\bbo) = \{(s^m,F(s)) \, : \, s \in (\C,0)\}.
\end{equation}
More precisely, writing $F = (f_2,\ldots,f_n)$, for each $j=2,\ldots,n,$ we 
know that 
$$
f_j (s) = s^{a_j} \phi_j(s), \;\; {\rm with} \;\; \phi_j(0) \neq 0, 
\;\;  {\rm and} \;\; a_j \geq m+1 
$$
with the convention $a_j =\infty$ if and only if $f_j \equiv 0$. 
Moreover there exists no holomorphic map germ $H : (\C,0) \to 
(\C^{n-1}, \bbo)$ such that $F(s) = H(s^N)$, for a positive integer
$N\geq 2$ dividing $m$.

\bigskip
Let $(X,\bbo)$ be a curve germ of $(\Cn,\bbo)$. Let 
$(X_1,\bbo),\ldots,(X_s,\bbo),$ be its local irreducible components. 
For each $j=1,\ldots,s,$ let $L_j$ be the tangent cone of $X_j$ at $\bbo$,
thus a complex line through $\bbo$. 

For each $j=1,\ldots,s,$ the (orthogonal) projection $\pi_j 
:(X_j,\bbo) \to (L_j,\bbo)$ is a complex finite mapping and it induces the 
germ of a holomorphic cover $(X_j^*,\bbo) \to (L_j^*,\bbo)$ with 
$m_j := m(X_j,\bbo) \geq 1$ sheets.

\begin{remark} The length of any rectifiable curve of an Euclidean
space is taken w.r.t. the Euclidean metric.
\end{remark}

\medskip
The first principal result of this section is the following
\begin{proposition}\label{prop:local-LNE}
The germ of analytic curve $(X,\bbo)$ of $(\Cn,\bbo)$ admits a 
representative which is locally LNE at $\bbo$ if and only if the 
following conditions are satisfied: (i) $m_j = 1$ for each $j=1,\ldots,
s$, and (ii)  $L_j \cap L_k = \{\bbo\}$ for each pair $(j,k)$ such that 
$1\leq j < k\leq s$. In other words, the curve germ $(X,\bbo)$ is LNE if and only if it is a union of non-singular curve germs intersecting pairwise transversally.
\end{proposition}
\begin{proof}
Let $\lbd_1,\ldots,\lbd_t,$ be complex lines of $\Cn$ through the origin.
We define the following ``angle''
$$
\dlt((\lbd_j)_{j=1,\ldots,t}) := \min_{1\leq j < k \leq t} \left \{|\bu_j -
\bu_k | : \bu_j \in \lbd_j\cap\bS_1^{2n-1} \; {\rm and} \; \bu_k \in \lbd_k 
\cap \bS_1^{2n-1} \right \}.
$$
Condition (ii) holds true if and only if $\dlt((L_j)_{j=1,\ldots,s}) >0$.

\medskip\noindent
$\bullet$ \em Assume that (i) and (ii) are satisfied. \em

\smallskip
Condition (i) means that each $(X_j,\bbo)$ is a non-singular germ. 
Condition (ii) means that the lines $L_j$s are pairwise transverse. 
This is the setting of Lemma \ref{lem:transverse-submfds}.

\medskip\noindent
$\bullet$ \em Assume that $X$ admits a representative locally LNE at
$\bbo$. \em

First we show the following property.
\begin{claim}\label{claim:LNE-L_1=L_2}
Let $(Y,\bbo)$ be a curve germ whose tangent cone at $\bbo$ consists of 
the single line $L$. If $m(Y,\bbo) \geq 2$, then the germ 
$(Y,\bbo)$ cannot admit any representative
locally LNE at $\bbo$. 
\end{claim}
\begin{proof}
For simplicity let $m := m(Y,\bbo)$ and assume that $L = \C\times \bbo 
\subset \Cn$. Let $\pi_L : \Cn \to L$ be the projection onto $L$. 
Therefore the map germ induced by the restriction of $\pi_L$ to $Y^*$
$$
\pi : (Y^*,\bbo) \to (L^*,0)
$$
is a $m$-sheeted holomorphic covering. Let $x_1 \neq 0$ and let 
$\by_1,\ldots,\by_m$ be the points of $\pi^{-1}(x_1)$.

Assume that $m\geq 2$. Since $\by_j = (x_1,\ups_j)$ for each $j=1,\ldots, m$, 
using Puiseux Parametrization \eqref{eq:puiseux} for each irreducible 
component of $(Y,\bbo)$, we check there exists a positive integer $p$ such
that 
$$
|\ups_j | \leq const  |x_1|^\frac{p+1}{p}, \;\; \forall \;
1\leq j \leq m
$$
where $const$ is a positive constant we do not need to denote 
specifically.    

If to connect the points $\by_1 $ and $\by_2$ in $Y$ we go through  
$\bbo$, we find.
\begin{equation}\label{eq:too-long-X}
d_\inn^X(\by_1,\by_2) \geq |\by_1| + |\by_2|.
\end{equation}
When a rectifiable path $\gm :[0,1] \to Y^*$ connects the points $\by_1$ and 
$\by_2$, its projection $\gm_1$ onto $L$ is contained in $L^*$ and goes
from $x_1$ to $x_1$. Since $\gm_1$ is rectifiable and not contractible
in $L^*$, we deduce
$$
2|x_1| \leq \lgth(\gm_1) < \lgth(\gm).
$$
In any case we have obtained that $d_\inn^Y(\by_1,\by_2) \geq 2|x_1|$,
while $|\by_1 - \by_2| = o(|x_1|)$. Therefore the germ $(Y,0)$ cannot admit
a representative locally LNE at $\bbo$.
\end{proof}
Claim \ref{claim:LNE-L_1=L_2} implies that $m_j = 1$ for each $j=1,\ldots,s,$ 
and also that the tangent cones $L_j$ are pairwise transverse. 
\end{proof}

\medskip
A by-product of Proposition \ref{prop:local-LNE} is a criterion for a 
complex 
curve of a compact complex manifold to be LNE presented in Proposition
\ref{prop:compact} below.

\medskip
Let $(M,g_M)$ be a Riemannian manifold. Let $\psi:\cU \to 
B_2^m$ be a $C^1$ chart. The subset $U_1 := \psi^{-1}(\bB_1^m)$ is a 
compact sub-manifold with boundary of $\cU$ and of $M$. Let 
$$
\be_1 := \psi^*(eucl|_{\bB_1^m})
$$
the pull-back of the Euclidean metric restricted to $\bB_1^m$. Since 
$\psi_1 :=\psi|_{U_1} : (U_1,d_{U_1}) \to (\bB_1^m,|-|)$ is bi-Lipschitz
(it is $C^1$), the metric space structures obtained from the
Riemannian metrics $h_1:=g_M|_{U_1}$ and $\be_1$ are 
equivalent. Thus $U_1$ is LNE in $(M,g_M)$.
\begin{lemma}\label{lem:localLNE-same}
The manifold $M$ is locally LNE (w.r.t. $g_M$) at $\bx_0 \in \cU$ 
if and only if $\psi(\cU)$ is locally LNE at $\bbo$ (w.r.t. $eucl$).  
\end{lemma}
\begin{proof}
We can assume that $\psi(\bx_0) = \bbo$.
The previous observation about $h_1$ and $\be_1$, the compactness 
of $U_1$ and Lemma \ref{lem:locally-LNE} altogether imply the statement.
\end{proof}
\begin{definition}\label{def:complex-curve}
Let $M$ be a complex manifold. A complex curve of $M$ is a complex 
analytic subset of $M$, thus closed, and which is of local (complex) 
dimension $1$ at each of its points.
\end{definition}
When the complex manifold $M$ is compact, any complex curve $X$ of $M$ is 
compact. Therefore the set $X_\sing$ of singular points of $X$ is either empty
or consists of finitely many points.  

\medskip
Following point (ii) of Remark \ref{rem:compact-manifold}, since any 
compact complex manifold can be equipped with a $C^\infty$ Hermitian 
structure, we will not refer to the Riemannian structure in the next 
statement, the other main result of this section.
\begin{proposition}\label{prop:compact}
A complex curve $X$ of a compact complex manifold $M$ is LNE if and only if 
$X$ is connected and is locally LNE at each point of its singular locus.
\end{proposition}
\begin{proof}
If $X$ is LNE, it is connected and it is locally LNE at each of its points
by Definition \ref{def:loc-LNE}.

\medskip
Let $m$ be the complex dimension of $M$.
Assume that $X$ is connected and is locally LNE at each of its singular 
points. Let $X_\sing := \{\bx_1,\ldots,\bx_s\}$. 
For each $j=1,\ldots,s,$ let 
$$
\psi_j :\cU_j \to \cV_m := B_2^{2m} \subset \C^m
$$
be a complex chart of $M$ centred at $\bx_j$ such that $\psi_j(\bx_j) = 
\bbo$. The image $Y_j := \psi(X\cap \cU_j)$ is a complex analytic set of 
$\cV_m$ of local complex dimension $1$ at each of its points. There
exists a positive radius $r_j < 2$ such that for every positive radius
$r\leq r_j$, the curve  
$$
Y_j\cap \bS_r^{2m-1} 
$$
is a non-singular real analytic compact curve consisting of $d_j\geq 1$ 
connected components (see \cite{Mil} for instance). Let 
$$
r_X := \min_{j=1,\ldots,s} r_j.
$$
We can assume that each intersection $\clos(\cU_k) \cap \clos(\cU_l)$ is 
empty for any pair $k,l$ with $k\neq l$. From Lemma \ref{lem:localLNE-same}, 
we can also assume that for each $j=1,\ldots,s,$ and for each positive 
radius $r\leq r_X$ the compact subset 
$$
X_{j,r} := X\cap \psi_j^{-1}(\bB_r^{2m})
$$
is LNE. By choice of $r_X$ we deduce that for each positive radius $r\leq 
r_X$, the subset 
$$
E_r := X\setminus \cup_{j=1}^s \psi_j^{-1}(B_r^{2m})
$$
is a compact real analytic sub-manifold with boundary 
$$
\dd E_r = \cup_{j=1}^s (X \cap \psi_j^{-1}(\bS_r^{2m-1})).
$$
Thus any connected component of $E_r$ is LNE. Observe that
for any $r \leq \frac{2r_X}{3}$ the subset
$$
T_{j,r} := X \cap \psi_j^{-1}(\bB_r^{2m} \setminus B_\frac{r}{2}^{2m})
$$
is a compact real analytic sub-manifold with boundary, consisting of $d_j$ 
connected components, each diffeomorphic to the cylinder $\bS^1 \times 
[\frac{r}{2},r]$. In particular each connected component of $T_{j,r}$ is
LNE. 
Since any point of $X$ is an interior point of $X_{j,r}$, $T_{j,r}$ or 
$E_r$, for some radius $r \leq \frac{2}{3} r_X$, $X$ is locally LNE at
each of its points by Lemma \ref{lem:localLNE-same}.
We  conclude the proof applying again Lemma \ref{lem:locally-LNE}.
\end{proof}
%
%
%
%
%
%
%
%
%
%
%
%
%
%
%
%
%
%
%
%
%
%
%
%
%
%
%
%
%
%
%
%
%
%
%
%
%
%
%
%
%
%
%
%
%
%
%
%
%
%
\section{Preliminaries at infinity}\label{section:PICn}
Let $\Cn$ be the complex affine space of dimension $n$ and let $\bx$
be some affine complex coordinates.

\begin{convention}
Unless mentioned otherwise, being LNE is to be 
understood as being LNE as a subset of $\Cn$ w.r.t. the Euclidean metric. 
\end{convention}

\medskip
We will need the following result, consequence of Lemma 
\ref{lem:LNE-complement}
\begin{claim}\label{claim:Cn-Br}
For any positive radius $R$, the subsets $\Cn \setminus B_R^{2n}$ and 
$\Cn \setminus \bB_R^{2n}$ are LNE.
\end{claim}

\bigskip
Let $\bP^n := \field{CP}^n$ be the complex projective space with projective
coordinates $[z_1:\ldots:z_{n+1}]$. Let $\bH_\infty$ be the projective 
hyperplane $z_{n+1}=0$.

The affine space $\Cn$ embeds in $\bP^n$ as $\bx  \mapsto [\bx:1]$, that
is $\bP^n = \Cn \sqcup \bH_\infty$. The hyperplane $\bH_\infty$ is called
the hyperplane at infinity and any point of $\bH_\infty$ is called \em a 
point at infinity. \em

If $S$ is any subset of $\bP^n$, the \em affine part $S^a$ of $S$ \em is 
defined as
$$
S^a := S\cap \Cn = S \setminus \bH_\infty.
$$
The subset $S$ is \em affine \em if $S = S^a$.

When $S$ is affine, the germ $(S,\infty) := (S,\bH_\infty)$ is well defined
as a germ in $\Cn$, and so is the germ $(S,\lbd)$ for $\lbd \in \bH_\infty$.

\medskip
Let $\lbd$ be a point of $\bH_\infty$. If $(Y,\lbd)$ is a (complex analytic) curve germ at $\lbd$, we will denote by $Y$ any representative of the germ $(Y,\lbd)$ in $\bP^n$. 
We are only interested in representatives of $(Y,\lbd)$ whose intersection 
with $\bH_\infty$ is just $\lbd$, property which we will always implicitly 
assume true.

Assume that $(Y,\lbd)$ is irreducible at $\lbd$ and is not contained in 
$\bH_\infty$. Thus the tangent cone to $Y$ at $\lbd$ is a line 
$L$ of the vector space $T_\lbd \bP^n$. The line $L$ is either transverse
to $\bH_\infty$ or contained in  $T_\lbd \bH^\infty$.
After a linear change of coordinates in $\Cn$, we can assume that $\lbd := 
[1:0:\ldots:0]$. Let $\cA_1$ be the affine chart of $\bP^n$ defined as 
$\{z_1\neq 0\}$. Let $[1:\bw:z]$ be affine coordinates in
$\cA_1$ so that in $\bH_\infty\cap\cA_1 = \{z=0\}$. We further write
$\bw = (w_2,\bw') \in \C\times\C^{n-2}$.
We identify $T_\lbd \bP^n$ with $\cA_1$ as well.

Since $\bx$ are the affine coordinates of $\Cn$, in $\Cn\cap \cA_1$
we write $[x_1:\by:1] = [\bx:1] = [1:\bw:z]$. 

If $L$ is contained in $\{z=0\}$, after a linear transformation of $\{z=0\}$ 
in $\cA_1$, corresponding to a linear transformation of $\{x_1 = 0\}$ in 
$\Cn$, we can further require that 
\begin{equation}\label{eq:L-tangent}
L = \{[1:w_2:0:\ldots:0] \, : \,  w_2\in \C\}.
\end{equation}
The line $L$ is transverse to $\bH_\infty$ at $\lbd$ if and only if there 
exists $\ba\in \C^{n-1}$ such that 
$$
L = \{ [1:z \ba:z] \, : \, z\in \C\}.
$$
The affine translation $\bx \to (x_1,\by - \ba)$ in  $\Cn$, corresponding
to the linear change of coordinates $(\bw,z) \to (\bw',z) = (\bv,z) = 
(\bw - z\ba,z)$ in  $\cA_1$, shows that the line $L$ has the following 
equation
\begin{equation}\label{eq:L-transverse}
L = \{[1:0:\ldots:0:z] \, : \, z \in \C \}. 
\end{equation}
Let $(s,\bw',t)$ be either $(w_2,\bw',z)$ or $(z,\bw',w_2)$ for $w_2,z,$ 
as appearing in the simplified equations 
\eqref{eq:L-tangent} and \eqref{eq:L-transverse}
of $L$, so that $L = \{t = \bw' 
= 0\}$. 

The multiplicity $m := m(Y,\lbd)$ is preserved by this change of 
coordinates to obtain the simpler equation of $L$. Let 
$$
Y^* := Y\setminus \lbd \;\; {\rm and} \;\; L^* := L\setminus \lbd.
$$
The projection $p_L :(Y^*,\lbd) \to (L^*,\lbd)$ is a holomorphic 
$m$-sheeted covering. 

\medskip
We introduce the following avatar at infinity of being locally LNE at a point
(see also \cite{FeSa}).
\begin{definition}\label{def:LNE-infty}
A subset $S$ of $\Cn$ is \em locally LNE at infinity \em (w.r.t. the 
euclidean metric of $\Cn$) if there exists a compact subset $K$ of $\Cn$ such 
that $S \setminus K$ is LNE (w.r.t. the 
euclidean metric of $\Cn$).
\end{definition}
The second non-compact result about being LNE is the following
\begin{lemma}\label{lem:LNE-transverse}
Assume $L$ is transverse to $\bH_\infty$. The germ $(Y^a,\lbd)$ is locally LNE 
at infinity if and only if $m=1$.
\end{lemma}
\begin{proof}
In this case we have $(s,t) = (z,w_2)$, and using parametrization 
\eqref{eq:puiseux} we find
$$
(Y,\lbd) = \left\{[1:F(z):z^m] = [1:O(z^p):z^m] \, : \, z\in (\C,0) 
\right\}
$$
for a positive integer $p\geq m+1$, maximal for this property.
In this case $F \equiv 0$ would mean $p=\infty$, which corresponds to
the case $Y^a = L$.
Observe also that from the affine point of view we check that 
$$
(Y^a,\lbd) = \left\{\left(x_1^m,x_1^m F(x_1^{-1})\right) =
\left(x_1^m, O\left(x_1^{-(p-m)}\right)\right) \, 
: \, x_1 \in (\C,\infty) \right\}
$$
so that if $(x_1,\by) \in (Y^a,\infty)$, then $\by \to \bbo$ as $x_1 \to 
\infty$. Let us consider the following holomorphic map germ $(\C,\infty)
\to (\C^{n-1},0)$ 
$$
x_1 \to G(x_1):= x_1^m F(x_1^{-1}) =  O\left(x_1^{-(p-m)}\right). 
$$

\medskip\noindent
$\bullet$ Assume $m=1$. The germ $(Y,\lbd)$ is non-singular at $\lbd$,
$$
\bw = z^p\aph(z) 
$$
where either $p\geq 2$ and $\aph(z)\in\C\{z\}^{n-1}$ with $\aph(0) \neq 
\bbo$, or $(Y,\lbd) = (L,\lbd)$.
This presentation is equivalent to the following affine one 
$$
\by = G(x_1).
$$
Since the components of $G$ are power series in $x_1^{-1}$ converging for 
$|x_1| \geq R_1>0$,  $G$ goes to $\bbo$ 
as $x_1$ goes to infinity. Thus its derivative goes to $\bbo$
as $x_1$ goes to infinity, therefore
it is a Lipschitz mapping $\C\setminus \bB_R^{2} \to \C^{n-1}$ for any radius
$R\geq R_1$. Proposition \ref{prop:lipschitz-graph} and Claim 
\ref{claim:Cn-Br}
imply that the subset $\{(x_1,G(x_1) \, : \; |x_1| \geq R\}$ is LNE 
since it is a Lipschitz graph over a LNE subset. Therefore the germ 
$(Y^a,\lbd)$ is locally LNE at infinity.

\medskip\noindent
$\bullet$ Assume $m\geq 2$. 
Let $z\in \C^*$ and let $\zt$ be a $m$-th root of $z$. Let 
$\bx_0 := [1:F(\zt):z]$ and $\bx_1 := [1:F(\omg \zt):z]$ for 
$\omg = e^{i\frac{2\pi}{m}}$. 
By parametrization \eqref{eq:puiseux}, we find $F(\zt) \neq F(\omg\zt)$.  
Writing 
$$
\bx_0^a := \left( \frac{1}{z},\frac{F(\zt)}{z} \right) \in \Cn \;\;
{\rm and} \;\; \bx_1^a := \left( \frac{1}{z},\frac{F(\omg\zt)}{z} \right)
\in \Cn 
$$
for $|z|$ small enough we obtain
\begin{equation}\label{eq:x1-x0}
|\bx_1^a - \bx_0^a| =  \left|\frac{F(\zt)}{z} - \frac{F(\omg\zt)}{z}
\right | \leq const. |z|^{\frac{p}{m} - 1} 
= const. \frac{1}{|x_1|^q} \;\; {\rm for} \;\; q:= \frac{p-m}{m} > 0.
\end{equation}
We can take a representative of $Y^a$  outside a large Euclidean 
closed ball $\bB^{2n}$ centred at $\bbo$ such that the projection
mapping
$$
\pi : Y^a \to \C\setminus \bB^2, \;\; [x_1:\by:1]  \mapsto x_1
$$
is a $m$-sheeted holomorphic covering, where $\bB^2$ is an Euclidean
ball of $\C$ centred at $\bbo$ image of $\bB^{2n}$ under the projection
onto the $x_1$-axis.

\smallskip
Let $\gm : [0,1] \to Y^a$ be a rectifiable path from $\bx_0:=\gm(0)$
to $\bx_1:=\gm(1)$. We write $\gm(t) = (x_1(t),\by(t))$.
The path $\pi \circ \gm: t\mapsto x_1(t)$ 
is well defined over $[0,1]$, and is a loop since $\pi(\bx_0) = \pi( \bx_1) 
\in \C^*$. Therefore it is not contractible in $\C \setminus \bB^2$.
We obtain   
$$
\lgth(\pi\circ\gm) := \int_0^1 |x_1'(t)| \rd t \, \leq \, \lgth(\gm). 
$$
Since $\pi\circ\gm$ is not contractible in $\C \setminus \bB$
the following estimate is satisfied
$$
|x_1(0)| \gg 1 \; \Longrightarrow \; 
|x_1(0)| \; \leq \; \lgth(\pi\circ\gm) \; \leq \; \lgth(\gm).
$$
Combining this last estimate with Estimate \eqref{eq:x1-x0}, which implies
that $|\bx_1^a - \bx_0^a|$ goes to $0$ as $|x_1(0)|$ goes to $\infty$, shows that
no representative of the germ $(Y^a,\lbd)$ can be locally LNE at infinity.
\end{proof}
The next result presents the case of tangency to $\bH_\infty$. Its proof will
follow from arguments similar to those used for the transverse case (see also
\cite{FeSa} for a similar situation). 
\begin{lemma}\label{lem:notLNE-tangent}
If $L$ is contained in $T_\lbd \bH_\infty$, then the germ $(Y^a,\lbd)$ is 
not locally LNE at infinity.
\end{lemma}
\begin{proof}
In this case we have $(s,t) = (w_2,z)$. 
Using again parametrization \eqref{eq:puiseux} we find
$$
(Y,\lbd) = \left\{[1:w_2^m:F(w_2)] = [1:w_2^m:O(w_2^{m+1})] \, : \, 
w_2\in (\C,0) \right\}.
$$
Since the curve germ is not contained in $\bH_\infty$, there exist a
positive integer $p\geq m+1$, a function $\vp \in \C\{w_2\}$ 
satisfying $\vp(0) \neq 0$ such that
$$
z(w_2) = w_2^p\cdot \vp(w_2),
$$
writing $F = (\bw',z) = (w_3,\ldots,w_n,z)$. For each $j=3,\ldots,n,$ there
exists an exponent $q_j \in\N_{\geq m+1}\cup\infty$ and a function $\psi_j \in 
\C\{w_2\}$ 
such that
$$
w_j(w_2) = w_2^{q_j} \cdot \psi_j(w_2)
$$
where $\psi_j(0) \neq 0$ if the function $w_j$ is not the null function,
equivalently if $q_j < \infty$. 
Since $\vp(0) \neq 0$, we can re-parametrize the affine part $(Y^a,\lbd)$ as
$$
(Y^a,\lbd) = \left\{[s^p:O(|s|^{p-m}):1]  \, : \, 
s \in (\C,\infty) \right\}.
$$
From the affine point of view, since $p-m > p - q_j$ for all $j$, we get
$$
(x_1,\by) \in Y^a \; \Longrightarrow \; 
|\by| \leq const. |x_1|^r \;\; {\rm for} \;\; 
r:= \frac{p-m}{p} \in ]0,1[
$$
when $|x_1|$ is large enough.

We can take a representative of $Y^a$ outside a large Euclidean closed 
ball $\bB^{2n}$ centred at $\bbo$ so that, by parametrization 
\eqref{eq:puiseux}, the projection mapping 
$$
\pi : Y^a \setminus \bB^{2n} \to \C\setminus \bB^2, \;\; 
[x_1:\by:1] \mapsto x_1
$$
is a $p$-sheeted holomorphic covering whenever $|x_1|$ is large enough, 
where $\bB^2$ is Euclidean ball of $\C$ centred at $\bbo$ image of 
$\bB^{2n}$ under the projection onto the $x_1$-axis. Let $\bx,\bx'$ in 
$Y^a$ such that 
$$
\pi(\bx) = \pi (\bx') = x_1.
$$
Since $p\geq m +1 \geq 2$, we assume that $\bx \neq \bx'$.
If $|x_1|$ is large enough, we find 
$$
|\bx|, \, |\bx'| = |x_1|(1+ o(1))
$$
as well as the following estimate
$$
|\bx - \bx'| = o(|x_1|)  
\ll |x_1|.
$$
From here on, we conclude as in the proof of the case $m\geq 2$ of Lemma
\ref{lem:LNE-transverse}.
\end{proof}
%
%
%
%
%
%
%
%
%
%
%
%
%
%
%
%
%
%
%
%
%
%
%
%
%
%
%
%
%
%
%
%
%
%
%
%
%
%
%
%
%
%
%
%
%
%
\section{Unbounded part of affine curves}\label{section:UPoAC}
This section continues what was initiated in Section \ref{section:PICn},
but working with explicit representatives of germs of infinity.

\begin{convention} A curve in $\Cn$ is an affine algebraic curve. 
A curve in $\bP^n$ is a projective algebraic curve.
\end{convention}

\bigskip
Let $X$ be a curve of $\bP^n$ and let $X^\infty$ be the intersection 
$X\cap \bH^\infty$. Let $\deg(X)$ be the degree of $X$. 
\begin{remark}\label{rmk:mult-inters-numbaffine-curve-LNE-infty}
1) If $X$ has no irreducible component contained in $\bH_\infty$, then 
$\deg(X) = \deg(X^a)$. 
\\
2) As a consequence of B\'ezout Theorem, if $card(X^\infty) = \deg(X)$, 
then for each $\lbd\in X^\infty$, the germ $(X,\lbd)$ is non-singular at 
$\lbd$ and is transverse to $\bH_\infty$ at $\lbd$.
\end{remark}
We assume that $X^\infty$ consists of finitely many points 
$\lbd_1,\ldots,\lbd_p,$ with $p = card (X^\infty) \leq \deg(X)$. 

\medskip
The germ of the affine part $X^a$ at infinity is 
$$
(X^a,\infty) = \sqcup_{j=1}^p (X^a,\lbd_j).
$$
Given any positive radius $R$, let 
$$
X_R := X \cap \bS_R^{2n-1} = X^a \cap \bS_R^{2n-1}.
$$
Using the Euclidean inversion $\iota_{2n} :\bx \mapsto 
|\bx|^{-2}\cdot\bx$, the Euclidean closure of the image of $\iota_{2n}
(X^a\setminus \bbo)$ is a real algebraic 
set of $\R^{2n} \supset \iota_{2n}(\Cn)$, with an isolated singularity at 
$\bbo$. The Local Conic Structure at $\bbo$ when combined with the
inversion yields a Conical Structure Theorem at infinity, which states 
that there exists a positive radius $R_X$ such that for $R,R'\geq R_X$
the links $X_R$ and $X_{R'}$ are $C^\infty$ diffeomorphic. Moreover
for $R\geq R_X$ the sub-manifold with boundary 
$$
X_\gR^a := X^a \setminus B_R^{2n}
$$ 
is $C^\infty$ diffeomorphic to the cylinder $X_R \times [R,+\infty[$, mapping 
$X_r$ onto $X_R \times r$. Let $\cC$ be a connected component of $X_\gR^a$. 
Therefore the closure of $\cC$ in $\bP^n$ intersects the hyperplane at 
infinity $\bH_\infty$ in a single point $\lbd_\cC$. Two such connected 
components 
$\cC_1$, $\cC_2$ may accumulate at the same point at infinity: 
$\lbd_{\cC_1} = \lbd_{\cC_2}$.

\medskip
For $R\geq R_X$, let $\cC_1^R,\ldots,\cC_e^R,$ be the connected components of
$X_\gR^a$. Thus $card(X^\infty) \leq e\leq \deg(X)$.
A first key piece to the main result is the following
\begin{proposition}\label{prop:LNE-unbounded}
Let $X^a$ be an affine curve and let $X$ be its projective closure. Assume 
that $\deg(X) = card(X^\infty)$. Then, there exists a positive radius $R_X'$
$\geq R_X$ such that for each $R\geq R_X'$, the connected components 
$\cC_j^R$ are LNE, $j=1,\ldots,\deg(X)$.
\end{proposition}
\begin{proof}
Let $d := \deg(X)$.
Let $i(X,\lbd)$ be the local intersection multiplicity of the germ $(X,\lbd)$
with $(\bH_\infty,\lbd)$ (see \cite{Dra}). 
Since $\bH_\infty$ is non-singular, we deduce that 
$i(X,\lbd) \geq m(X,\lbd)$ for any $\lbd\in X^\infty$. The hypothesis means
that 
$$
d = \sum_{\lbd\in X^\infty} i(X,\lbd),
$$ 
and thus $1 = i(X,\lbd) \geq m(X,\lbd) \geq 1$. Therefore the curve germ 
$(X,\lbd)$ is non-singular and is in general position with $\bH_\infty$ at 
$\lbd$ for each $\lbd\in X^\infty$, that is transverse to $\bH_\infty$. 
Thus we are in the hypotheses of Lemma \ref{lem:LNE-transverse}.

Let $X^\infty = \{\lbd_1,\ldots,\lbd_d\}$. 
Assume that $(X^a,\lbd_j) = (\cC_j^R,\lbd_j)$ for each $j=1,\ldots,d,$
and each $R\geq R_X$. 

Let $L_j$ be the line tangent to $X$ at $\lbd_j$. Since it is transverse
to $\bH_\infty$ we can consider $L_j$ as an affine line of $\Cn$. Let
$\be_j$ be a unit vector of $\Cn$ such that the point $\lbd_j \in \bH_\infty$
corresponds to the complex line $\C\be_j$. 
Let $H_j$ be the complex hyperplane orthogonal to $\C\be_j$. 
Since $L_j$ is transverse to $\bH_\infty$ there exists $\ba_j \in H_j$ such 
that
$$
L_j = \C\be_j + \ba_j.
$$
\begin{claim}\label{claim:bilipschitz}
$\cC_j^R$ is bi-Lipschitz homeomorphic to $L_j\setminus B_R^{2n}$ with respect to the outer distance.
\end{claim}
\begin{proof}[Proof of the claim]
Similarly to the proof of Lemma \ref{lem:LNE-transverse}, we deduce that 
$(\cC_j^R,\lbd_j)$ is a Lipschitz and holomorphic graph over
$\C \setminus B_{r_j}^2 = \C\be_j\setminus B_{r_j}^2$ for each 
$j=1,\ldots,d$. Moreover the Lipschitz constant tends to $1$ as $r_j$ goes
to $\infty$. Let 
$$
r_X:= \max_{j=1,\ldots,d} r_j.
$$
For each $R\geq r_X$, let 
$$
Y_j^R := \{(s,G_j(s)) = s\be_j + G_j(s) \in \C\be_j \times H_j: 
s\in \C\setminus B_R^2\}
$$
where $G_j$ is the holomorphic mapping taking values in $\C^{n-1} = H_j$ 
we mentioned. We recall that 
$$
G_j(s) = g_j \left(\frac{1}{s}\right) \;\; {\rm where} \;\; 
g_j(t) \in \C\{t\}^{n-1} \;\; {\rm and} \;\; g_j(0) = \ba_j.
$$
Therefore, the following estimate holds true
$$
\lim_{|s|\to \infty} \; G_j'(s) = \bbo.
$$
Since $G_j$ goes to $\ba_j$ and its derivative goes to $\bbo$ as $s$ goes
to $\infty$, for every 
positive $\ve$, there exists $r_\ve \geq \max (R_X,r_X)$ such that 
$$
x_1,x_1' \in \C\be_j \setminus B_{r_\ve}^2 \; \Longrightarrow \; 
|G_j(x_1) - G_j(x_1')| < \ve |x_1 - x_1'| \;\; {\rm and} \;\;
|G_j(x_1)|^2 < |\ba_j|^2+\ve^2.
$$
Thus the mapping $\Gm_j: s \mapsto s\be_j + G_j(s))$ is holomorphic
over $\C \setminus B_{r_\ve}^2$ and is Lipschitz, with Lipschitz constant 
smaller than $(1+\ve)$. 

Since $\Gm_j$ is bi-holomorphic, up to taking a 
smaller $\ve$ we find that $p_j := \Gm_j^{-1}$, that is the projection onto
$\C\be_j$, is Lipschitz with constant larger than $(1-\ve)$. Natural 
embedding of $\C\be_j \setminus B_{r_\ve}^2$ into $\C^n$ yields the 
first part of the proof. Last, there is just to recall that the orthogonal 
projection onto an affine space is Lipschitz.
\end{proof}

In the notation of the proof above, if $\ve$ and $r_\ve$ are given, we find that for 
each $j=1,\ldots,d,$ the following inclusions are satisfied
$$
Y_j^{R_{j,\ve}} \; \subset \;\cC_j^{R_{j,\ve}} \; \subset \;  Y_j^R,
\;\; {\rm where} \;\; R_{j,\ve} := \sqrt{R^2 + |\ba_j|^2 + \ve^2}
$$
whenever $R\geq r_\ve$. Under this choice of $\ve$ we get
$$
\dd \cC_j^{R_{j,\ve}} \cap Y_j^{R_{j,\ve}} = \dd Y_j^R \cap 
\cC_j^{R_{j,\ve}} = \emptyset.
$$
where $\dd \cC_j^{R} = \cC_j^{R}\cap \bS_{R}^{2n-1}$ and $\dd Y_j^R = 
\Gm_j(L_j\cap \bS_R^{2n-1})$. We define the following smooth Jordan curve of
$L_j$
$$
Z_j^R := p_j(\dd\cC_j^R),
$$
which is diffeomorphic to $\bS^1$. Let $D_j^R$ be the open disk it bounds. We 
cannot expect $D_j^R$ to be convex.  

\smallskip
The following claim follows from Lemma \ref{lem:LNE-complement}.
\begin{claim}\label{claim:L_j-DjR-LNE}
$\C\be_j\setminus D_j^R$ is LNE.
\end{claim}

Since $\bB_R^2\cap \C\be_j \subset D_j^{R_{j,\ve}}$ and the mapping $\Gm_j$ 
is Lipschitz over $\C\be_j\setminus B_R^2$, it is also Lipschitz over 
$\C\be_j \setminus D_j^R$. The result follows from Proposition 
\ref{prop:lipschitz-graph}. This concludes the proof of Proposition \ref{prop:LNE-unbounded}.
\end{proof}

\medskip
There are two scenarii for which $card(X^\infty) < \deg (X)$. The first one
has already been mentioned in Lemma \ref{lem:notLNE-tangent}, when the germ 
$(X,\lbd)$ is irreducible and with tangent cone $L$ at $\lbd$  contained 
in the hyperplane at infinity. The second is when $(X,\lbd)$ is not irreducible.
\begin{lemma}\label{lem:notLNE-infty-local}
Let $X$ be a curve of $\bP^n$. Let $\lbd$ be an isolated point of $X^\infty$. 
If the germ $(X,\lbd)$ is not irreducible, then $X^a$ is not LNE.
\end{lemma}
\begin{proof}
We can assume that $\lbd = [1:0:\ldots:0]$. 
Let $\left\{(X_j,\lbd)\right \}_{j=1,\ldots,N}$ be the irreducible components
of $(X,\lbd)$. 
Fix $R\geq R_X$ so that $S:=X_\gR^a$ is not connected.
Note that 
$$
(X^a,\lbd) = (S,\lbd).
$$
For each connected components $\cC$ of $S$ accumulating at $\lbd$, 
there exists a unique $j\in\{1,\ldots,N\}$ such that 
$$
(\cC,\infty) = (X_j^a,\lbd).
$$
Let $\cC_1$, $\cC_2$ be two connected components of $S$ whose germ at 
$\infty$ are $(X_1^a,\lbd),(X_2^a,\lbd)$, respectively.
Let $(\bx_l)_l \subset \cC_1$ and $(\bx_l')_l \subset \cC_2$ be two sequences 
converging to $\infty$. We further require that 
$$
\bx_l = (l,\bx_k) \;\; {\rm and} \; \; \bx_l' = (l,\bx_k').
$$
Thus we obtain the following estimate
$$
|\bx_l - \bx_l'| =o(l).
$$
Let $\gm_l$ be a rectifiable path on $X^a$ connecting $\bx_l$ and 
$\bx_l'$. Its length satisfies the following inequality  
$$
\lgth (\gm_l)  \; \geq \;  {\rm dist} (\bx_l,\dd S) + {\rm dist} (\bx_l',
\dd S),
$$
where $\dd S$ is the boundary of $S$, which is compact. 
For $l$ large enough we check that 
$$
{\rm dist} (\bx_l,\dd S)  \geq \frac{l}{2} \;\; {\rm and} \;\;
{\rm dist} (\bx_l',\dd S) \geq \frac{l}{2}.
$$
Since for $l$ large enough we have found that 
$$
\lgth (\gm_l)\; \geq \; l,
$$
the subset $X^a$ cannot be LNE.
\end{proof}
\begin{remark}
The demonstration of Lemma \ref{lem:notLNE-infty-local} still holds whenever 
a germ $(S,\infty)$ of a closed connected subset $S$ of $\R^q$ has 
two connected components at infinity,
 whose respective accumulation loci at infinity contain a common point.
\end{remark}
%
%
%
%
%
%
%
%
%
%
%
%
%
%
%
%
%
%
%
%
%
%
%
%
%
%
%
%
%
%
%
%
%
%
%
%
%
%
%
%
%
%
%
%
\section{Bounded part of affine curves}\label{section:BPoAC}
We continue the investigations of Section \ref{section:UPoAC}, using
the exact same objects, notations, and hypotheses to complete what was done 
there in looking more closely this time at the ``bounded part'' of $X^a$.

\bigskip
Let $X_\sing$ be the singular locus of $X$, consisting at most of 
finitely many points. Observe that the affine part $X_\sing^a$ of $X_\sing$
is exactly $(X^a)_\sing$ the set of singular points of the affine part $X^a$.

For $R\geq R_X$ the following 
subset
$$
X_\lR^a := X^a \cap \bB_R^{2n}
$$
is a semi-algebraic subset of $\C^n$ with $C^\infty$ boundary $X_R$.
If the affine part $X^a$ is connected, then for $R\geq R_X$ the subset 
$X_\lR^a$ is also connected.

\medskip The next result is another key piece of our main result and 
is similar to Proposition \ref{prop:compact}.
\begin{proposition}\label{prop:LNE-bounded}
Let $X$ be a projective curve such that its affine part $X^a$ is connected
and not empty.
For $R\geq R_X$ the subset  $X_\lR^a$ is LNE if and only if $X^a$ is locally
LNE at each point of $X_\sing^a$.
\end{proposition}
\begin{proof}
Let $S:=X_\lR^a$, which is semi-algebraically-path-connected. Since 
$S\times S$ is connected, compact and semi-algebraic (thus rectifiable), 
the function $d_\inn^S$ admits a maximum $\ell_S$.

\medskip\noindent
$\bullet$ The condition of $S$ being locally LNE at each point of $X_\sing^a$ 
is certainly necessary to have $S$ LNE.

\medskip\noindent
$\bullet$ Assume that $X^a$ is locally LNE at each point of $X_\sing^a$.

Each point $\bx$ of $X^a\setminus X_\sing^a$ is one of the following two 
types: the germ $(S,\bx)$ is either bi-holomorphic to the germ $(\C,0)$ 
if $|\bx|<R$, or is $C^\infty$ diffeomorphic to the germ $(\R \times 
[0,\infty),\bbo)$ when $|\bx|=R$. In both situations, $S$ is locally LNE
at $\bx$ by Corollary \ref{cor:submfd-LNE}. Therefore $S$ is locally LNE
at each of its points. Since it is compact, Lemma \ref{lem:locally-LNE} 
yields that $S$ is LNE. 
\end{proof}
%
%
%
%
%
%
%
%
%
%
%
%
%
%
%
%
%
%
%
%
%
%
%
%
%
%
%
%
%
%
%
%
%
%
%
%
%
%
%
%
%
%
%
%
%
%
%
%
\section{Main Result: how to recognize an affine LNE curve}\label{section:MR}
We have gathered in the previous sections all the ingredients to prove the 
main result of the paper:
\begin{theorem}\label{thm:main}
Let $X$ be a connected projective curve of $\bP^n$ of degree $\deg (X)$ 
such that $X^\infty$ is finite. The affine curve $X^a$ is LNE if and only if the 
following conditions are satisfied: $(0)$ $X^a$ is connected; 
(i) $X^a$ is locally LNE at each of its singular points; (ii) The curve
$X$ intersects $\bH_\infty$ in exactly $\deg(X)$ points. 
\end{theorem}
\begin{proof}
Assume $X^a$ is LNE. It is connected and it is locally LNE at each of its 
points, 
thus we obtained $(0)$ and $(i)$. Since Lemma \ref{lem:notLNE-infty-local} cannot 
happen, by Lemma \ref{lem:notLNE-tangent} we deduce that at each $\lbd \in 
X^\infty$ the germ $(X,\lbd)$ is non-singular. 
Lemma \ref{lem:LNE-transverse} further implies that 
$(X,\lbd)$ is transverse to $\bH_\infty$ at each $\lbd$ of $X^\infty$.
Since $\deg(X)$ is the intersection number of $X$ with any projective 
hyperplane
of $\bP^n$, the transversality of $X$ and $\bH_\infty$ at each of their
intersection points yields $card(X^\infty) = \deg(X)$.

\medskip\noindent
Assume that conditions $(0)$, $(i)$, $(ii)$ are satisfied. 

Thus there exists $R_X'$ such that $X_\lR^a$ is connected for each 
$R\geq R_X'$. Since $X_R$ is a smooth compact sub-manifold once $R\geq R_X'$, 
hypothesis $(i)$ and Proposition \ref{prop:LNE-bounded} guarantee that
$X_\lR^a$ is LNE for each $R\geq R_X'$. Let 
$$
\cC_0 := X_\lR^a.
$$
Let $d =\deg(X)$ and let $X^\infty = \{\lbd_j\}_{j=1,\ldots,d}$.  
Condition $(ii)$ implies that Proposition \ref{prop:LNE-unbounded} is 
satisfied,
that is, each connected component of $X_\gR^a$ is LNE whenever $R\geq R_X'$. 
Let $\cC_1,\ldots,\cC_d,$ be the connected components of $X_\gR^a$, indexed 
such that $\lbd_j$ is the accumulation point at infinity of $\cC_j$. 

For each $j=0,\ldots,d,$ let $d_j := d_\inn^{\cC_j}$. 

Since each $\cC_j$ is LNE, let $A$ be a positive constant which is 
a LNE constant for each $\cC_j$:
$$
\bx,\bx'\in \cC_j \; \Longrightarrow \; d_j(\bx,\bx') \; \leq \; 
A |\bx-\bx'|,  \;\; j=0,\ldots,d.
$$
\begin{claim}\label{claim:xj-xk}
There exists a positive constant $A'$ such that for any $0 \leq j<k\leq d$
the following estimate holds true
$$
\bx_j\in \cC_j, \bx_k \in \cC_k \; \Longrightarrow \; 
 d_\inn^{X^a} (\bx_j,\bx_k) \leq A' |\bx_j - \bx_k|.  
$$
\end{claim}
\begin{proof}[Proof of the claim]
For $j \geq 1$, let $\C_j$ be the complex line of $\Cn$ through the origin 
corresponding to complex line direction
$\lbd_j$. Let $S_j := \C_j \cap \bS_1^{2n-1}$ be the unit circle of $\C_j$ 
centred at the origin. Whenever $j\neq k$, the complex lines $\C_j$ and $\C_k$ 
only meet at $\bbo$, therefore the intersection $S_j \cap S_k$ 
is empty. For any $1\leq j  < k \leq d$, let 
$$
\dlt_{j,k} := \dist(S_j,S_k) > 0,
$$
the Euclidean distance between $S_j$ and $S_k$, and let 
$$
\dlt := \min_{1\leq j<k\leq d} \dlt_{j,k} >0.
$$
Since $\cC_j$ accumulates at $\lbd_j$ at infinity, we can assume that $R$ is 
large enough so that
$$
\sup \left\{ \dist\left(S_j,\frac{\bx}{|\bx|}\right) : \bx \in \cC_j \right\} \;
\leq \; \frac{\dlt}{4}
$$
for each $j=1,\ldots,d$. 

Let $\bx_j\in \cC_j$ and $\bx_k \in \cC_k$, for a given pair of indices  
$0\leq j<k \leq d$. Let $\bu_l \in \bS^{2n-1}$ (when defined) and 
$r_l \geq 0$, for $l=j,k,$ be defined as follows:
$$
\bu_l := \frac{\bx_l}{|\bx_l|}, 
\;\; {\rm and} \;\; r_l := |\bx_l|.
$$ 
Let $2\aph\in [0,\pi]$ be the non-oriented angle between $\bu_j$ and
$\bu_k$. 

\smallskip\noindent
$\bullet$ {\bf Case 1:} $1\leq j < k \leq d$.

Let $\bx_k$ be a point of $\dd \cC_l \subset X_R^a$ which realizes 
the minimum of $d_l (\bx_l,\dd\cC_l)$ for $l=j,k$. 

By choice of $R$ and $j,k$ we have 
$$
|\bu_j - \bu_k| = 2\sin(\aph)\; \geq \; \frac{\dlt}{2},
$$
Assume that $r_j \geq r_k$, we recall  identity  \eqref{eq:xj-xk-equal}
$$
|\bx_j - \bx_k|^2 = (r_j+r_k)^2\sin^2 \aph + (r_j-r_k)^2\cos^2\aph 
$$
from which we deduce again Estimates \eqref{eq:xj-xk-origin}
$$
\frac{\dlt}{4} \cdot (r_k+r_j) \leq |\bx_j - \bx_k| .
$$
Since $r_j, r_k \geq R$, Estimate \eqref{eq:xj-xk-inequal} yields the 
following estimate
\begin{equation}\label{eq:yj-yk}
|\by_j - \by_k| \; \leq \; 2R \; \leq \; \frac{4}{\dlt} \cdot |\bx_j - \bx_k|.
\end{equation}
as well as
\begin{equation}\label{eq:xl-yl}
|\bx_l - \bx_k| \; \leq \; r_l + R\; \leq \; \frac{4}{\dlt} \cdot 
|\bx_j - \bx_k|
\;\; {\rm for} \;\; l=j,k.
\end{equation}
Since, we obviously  have 
$$
d_\inn^{X^a}(\bx_j,\bx_k) \leq d_j(\bx_j,\by_j) + d_0(\by_j,\by_k) + 
d_k(\by_k,\bx_k), 
$$ 
combining Estimates \eqref{eq:yj-yk} and \eqref{eq:xl-yl} with $\cC_0,\cC_j,
\cC_k$ being LNE with LNE constant $A$ yield an expected inequality
$$
 d_\inn^{X^a}(\bx_j,\bx_k) \; \leq \; \frac{12A}{\dlt}\cdot|\bx_j -\bx_k|.
$$

\smallskip\noindent
$\bullet$ {\bf Case 2:} \em $j=0$ and $1\leq k \leq d$. \em

We can assume $r_0 < R < r_k$. 

If $r_k \leq 2R$, we know that $X_{\leq 2R}^a$ is LNE. 

Let $\by$ be a point of $\dd \cC_k$. 
When $r_k \geq 2R$, we find that 
$$
\frac{1}{2}r_k \; \leq \; |\bx_k - \bx_0 | \; \leq \; \frac{3}{2}r_k, \;\; 
{\rm and} \;\;
\frac{1}{2} r_k \; \leq \; d_k(\bx_k,\by) \; \leq \;  \frac{3A}{2}r_k.
$$
For $r_k\geq 2R $, we therefore deduce the following estimate
$$
 d_\inn^{X^a}(\bx_0,\bx_k) \; \leq \; d_0(\bx_0,\by) + d_k (\bx_k,\by) \; 
\leq \; 2 A R +  \frac{3A}{2} r_k
\; \leq \; 5A  |\bx_k - \bx_0|,
$$
proving the statement in the second case.
\end{proof}
Claim \ref{claim:xj-xk} establishes the desired LNE properties between any 
pair of points belonging to different subsets $\cC_j$.
Since each of these subset is LNE, the result is proved. 
\end{proof}
%
%
%
%
%
%
%
%
%
%
%
%
%
%
%
%
%
%
%
%
%
%
%
%
%
%
%
%
%
%
%
%
%
%
%
%
%
%
%
%
%
%
%
%
%
%
%
%
%
\section{Affinely LNE versus Projectively LNE}\label{section:AvsP}
The space $\bP^n$ is naturally equipped with a Kh\"aler metric, the 
Fubini-Study metric. As mentioned in point (ii) of Remark 
\ref{rem:compact-manifold}, any continuous Riemannian structure equipping 
$\bP^n$, Hermitian or not, will yield a metric space structure equivalent to 
the one induced from the Fubini-Study metric.
We will refer to any such continuous Riemannian metric on $\bP^n$ as the 
projective metric $g_\pro$.

\medskip
Let $H$ be a hyperplane of $\bP^n$, and let $\C_H^n$ be the affine space
$\bP^n \setminus H$. We can embed $\Cn$ in $\bP^n$ as $\C_H^n$ for any 
a priori given hyperplane $H$ of $\bP^n$ such that the embedding provides
a unitary isomorphism between $\Cn$ and $\C_H^n$. 

\medskip
As a subset of the projective space $\bP^n$, the affine space $\C^n$ is 
equipped with the restriction of the projective metric $g_\pro$, 
whereas as a complex vector space it is equipped with its canonical Hermitian 
structure $eucl$. The metric space structures respectively obtained
from $(\Cn,g_\pro)$ and $(\Cn, eucl)$ are clearly not equivalent, 
since $\C^n$ is bounded for $g_\pro$ while it is not bounded for $eucl$. 
Yet the respective restrictions of $eucl$ and $g_\pro$ to any given compact 
subset 
of $\Cn$ produce equivalent metric space structures over the compact subset, 
for the outer metric and thus also for the inner metric. 

\medskip
In this section, Theorem \ref{thm:main-proj} establishes the relation 
between LNE property of affine curve w.r.t. the projective and 
euclidean metrics. We present some conventions first.

\begin{definition}\label{def:gen-pos}
Let $X$ be a curve of $\bP^n$ and let $H$ be a hyperplane of $\bP^n$.
\\
1) The curve $X\setminus H$ in $\C_H^n$ is the \em affine trace of $X$ w.r.t. 
$H$. \em
\\
2) The pair $X,H$ is in \em general position \em if the intersection 
$X\cap H$ consists only of non-singular points of $X$ at each of which $X$ 
and $H$ are transverse.
\end{definition}
Here we forget about Convention 3, introducing
in its stead  the following
\begin{definition}\label{def:proj-LNE}
(i) A subset $S$ of $\Cn$ is LNE in $\Cn$ if it is LNE in $\C^n$ w.r.t.  
$eucl$.

\smallskip\noindent
(ii) A subset $S$ of $\bP^n$ is \em LNE in $\bP^n$ \em if it is LNE
in $\bP^n$ w.r.t. $g_\pro$. 
\end{definition}
%
%
Observe that when the curve $X$ is LNE in $\bP^n$, its affine trace 
$X\setminus H$ is not necessarily LNE in $\C_H^n$ for any hyperplane $H$,
as already suggested by Lemma \ref{lem:notLNE-tangent} and Lemma 
\ref{lem:notLNE-infty-local}.
\begin{example} Let $\bP^2$ be equipped with projective coordinates 
$\bz =[x:y:z]$. Let $L$ be the hyperplane $\{z=0\}$ and $\C^2:=\C_L^2$.
Consider the non-singular quadrics 
$$
P :=\{\bz : yz - x^2 = 0\} \;\; {\rm and} \;\; 
H := \{\bz : xy -z^2 = 0\}.
$$
The parabola $P$ and the hyperbola $H$ are isomorphic via the unitary 
linear change of coordinates $[x:y:z] \to [z:y:x]$.
Both curves are LNE in $\bP^2$ since they are non-singular and irreducible.
The affine traces w.r.t. $L$ are
$$
P\setminus L := \{(x,y) : y - x^2= 0\}  \;\; {\rm and} \;\; 
H\setminus L := \{(x,y) : xy - 1 = 0\}.
$$ 
By Theorem \ref{thm:main}, the hyperbola $H\setminus L$ is LNE in $\C^2$, 
while the parabola $P\setminus L$ is not LNE in $\C^2$. 
The difference being that $H,L$ are in general position while 
$P,L$ are not.
\end{example}
The next and penultimate result points at the very close interplay between
a projective curve $X$ and its general affine traces while investigating 
the LNE natures of $X$.
\begin{theorem}\label{thm:main-proj}
Let $X$ be a curve of $\bP^n$. The following statements are equivalent:
\\
(i) $X$ is LNE in $\bP^n$. 
\\
(ii) $X$ is connected and, at each singular point $\lbd$ of $X$ the germ 
($X,\lbd)$ is a finite union of pairwise transverse non-singular curve germs.
\\
(iii) For any hyperplane $H$ of $\bP^n$ such that the pair $X,H$ is in
general position, the variety $X \cup H$ is LNE in $\bP^n$.
\\
(iv) For any hyperplane $H$ of $\bP^n$ such that the pair $X,H$ is in
general position, the affine curve $X\setminus H$ is LNE in $\C_H^n$.
\\
(v) There exists a hyperplane $H$ of $\bP^n$ in general position with $X$
such that the affine curve $X\setminus H$ is LNE in $\C_H^n$.
\end{theorem}
\begin{proof}
The equivalence $(i)$ and $(ii)$ is simply Proposition \ref{prop:compact}.

\smallskip\noindent
Implication $(ii) \Longrightarrow (iii)$ follows from Lemmas \ref{lem:transverse-submfds} and \ref{lem:locally-LNE}.
 
\smallskip\noindent
To obtain $(iii) \Longrightarrow (i)$, note that if the union $X\cup H$
is LNE in $\bP^n$ then it is locally LNE in $\bP^n$ at each singular point
of $X$. Thus Proposition \ref{prop:compact} yields the implication.

\smallskip\noindent
To prove $(i)\Longrightarrow (iv)$, we observe that whenever $X$ and $H$
are in general position, the intersection $H \cap X_\sing$ is empty, and 
thus the affine trace $X\setminus H$ is connected. 
Therefore the affine trace $X\setminus H$ is LNE in $\C_H^n$ 
by point $(ii)$ and Theorem \ref{thm:main}.

\smallskip\noindent
Implication $(iv)\Longrightarrow (v)$ is obvious. 

\smallskip\noindent
To end the proof, let us show $(v) \Longrightarrow (i)$.
Let $\Cn := \C_H^n$ and let $X^a := X\setminus H$. 
By hypothesis any point of $X\cap H$ is a non-singular point of $X$, 
therefore $X_\sing =(X^a)_\sing$. Since $X^a$ is LNE in $\Cn$, it is 
locally LNE in $\Cn$ at each of its singular points, and so the affine curve 
$X^a$ contains a compact neighbourhood of each point of $X_\sing$ which is 
LNE in $\Cn$. Thus $X$ is locally LNE in $\bP^n$ at each of its singular 
points. We conclude applying Proposition~\ref{prop:compact}.
\end{proof}
%
%
%
%
%
%
%
%
%
%
%
%
%
%
%
%
%
%
%
%
%
%
%
%
%
%
%
%
%
%
%
%
%
%
%
%
%
%
%
%
%
%
%
%
%
%
%
\section{Lipschitz classification of LNE curves}\label{section:last}

In this section we present a Lipschitz complete invariant of LNE affine curves.

\begin{definition}\label{def:classify}
Let $X_1^a$ in $\C^{n_1}$ and $X_2^a$ in $\C^{n_2}$ be two affine curves.

1) They are \em topologically equivalent \em if there exists a homeomorphism 
$X_1^a \to X_2^a$.

2) They are \em outer Lipschitz equivalent \em if there exist a 
bi-Lipschitz homeomorphism 
$$
(X_1^a,d_{X_1^a}) \to (X_2^a,d_{X_2^a}). 
$$
\indent
3) They are \em inner Lipschitz equivalent \em if there exist a 
bi-Lipschitz homeomorphism
$$
(X_1^a,d_\inn^{X_1^a}) \to (X_2^a,d_\inn^{X_2^a}).
$$ 
\end{definition}

Obviously Lipschitz equivalence implies topological equivalence. Moreover, 
for LNE curves outer Lipschitz equivalence is the same as inner 
Lipschitz equivalence, thus in case of LNE curves we will speak simply of 
Lipschitz equivalence. We will show that in fact for LNE curves all three 
classifications coincide.

\bigskip
Let $X$ be an irreducible curve of $\bP^n$ of degree $d \geq 1$.
For each $\by \in X$, let $c(\by)$ be the number of 
irreducible components of the curve germ $(X,\by)$. The curve $X$
is bi-rationally equivalent to a compact Riemann surface $S_g$, which is 
a compact connected oriented surface of genus $g$.
Therefore the Euler Characteristic of $X^a\setminus X_\sing^a$ satisfies 
the following identity
\begin{equation}\label{eq:euler-affine-0}
\chi(X^a \setminus X_\sing^a) = 2 - 2g  - \sum_{\by\in X_\sing^a}c(\by)
- \sum_{\lbd \in X^\infty} c(\lbd).
\end{equation}

Given the irreducible LNE curve $X^a$ of $\Cn$ with projective closure
$X$, we define the following ordered object (encoding degree, number of 
singular points, genus, and multiplicities at singular points)
$$
DSGM(X):= \left(\deg(X); s(X); g(X); {(\by,c(\by))}_{\by \in 
X_\sing}\right),
$$
where $s(X) := \#X_\sing$ and $g(X)$ is the genus of the 
surface $S_g$ resolved from $X$. When $s(X) = 0$,
we suppress the last entry, that is $DSGM(X) = (\deg(X),0,g(X))$. Two objects 
$$
DSGM(X_i) = \left(d_i;s_i; g_i; {(\by,c(\by))}_{\by \in (X_i)_\sing}
\right), \;\; i=1,2,
$$
are equivalent if $d_1 = d_2$, $s_1 = s_2 =: s$, $g_1 = g_2$, and there 
exists a bijection $\vp$ of $\{1,\ldots,s\}$ such that $c(\by_j^1) = 
c(\by_{\vp(j)}^2)$ for each $j=1,\ldots,s$, where $(X_i)_\sing = 
\{\by_1^i,\ldots,\by_s^i\}$ for $i=1,2$.

\medskip
\begin{theorem}\label{thm:classify}
The following statements are equivalent:
\begin{enumerate}[(a)]
	\item Two irreducible complex affine algebraic LNE curves are Lipschitz equivalent.
	\item Two irreducible complex affine algebraic LNE curves are topologically equivalent.
	\item The $DSGM$ invariants of the two irreducible complex affine algebraic LNE curves are equivalent.
\end{enumerate}

\end{theorem}
\begin{proof}
Let $X_i^a$ be an irreducible LNE curve of $\C^{n_i}$, for $i=1,2$. Let $X_i$
be the projective closure of $X_i^a$ in $\bP^{n_i}$. 
Denote 
$$
d_i := \deg(X_i), \;\; Y_i := (X_i)_\sing = \{\by_i^1,\ldots,\by_i^{s_i}\},
\;\; {\rm and} \;\; X_i^\infty = \{\lbd_i^1,\ldots,\lbd_i^{d_i}\},
$$
since $X_i^a$ is LNE by hypotheses implies that 
$\#X_i^\infty = d_i$. 
\\
For $\by_i \in Y_i$, let $T_{\by_i} X_i^a$ be the tangent cone of 
$X_i^a$ at $\by_i$.
\\
Let $R\gg 1$ be a large radius, as in the proof of Lemma 
\ref{lem:transverse-submfds}, such that for each 
$\by_i^j\in X_i^a$, with $i=1,2$ and $j=1,\ldots,s_i$, the  subset
$$
\cC_i^j := X_i^a \cap B_i^j 
$$
is bi-Lipschitz homeomorphic to its truncated tangent cone
$$
T_{\by_i^j} X_i^a \cap B_i^j,
$$
where $B_i^j$ is either the open ball $B^{2n}(\by_i^j,R^{-1})$ or its closure
$\bB^{2n}(\by_i^j,R^{-1})$. 
In particular the boundary $\dd \cC_i^j$ is a finite union
of $c(\by_i^j)$ embedded $\bS^1$. 
Similarly, for $i=1,2$ and by Claim \ref{claim:bilipschitz},
once $R$ is large enough, the subset
$$
\cX_i := X_i^a \setminus \cB_i,
$$
is bi-Lipschitz homeomorphic to the truncated cone 
$$
\wh{X_i^\infty} \setminus \cB_i = \left(\cup_{j=1}^{d_i}\C\lbd_i^{j} \right)
\setminus \cB_i,
$$
where $\cB_i$ is either the open ball $B_R^{2n}$ or its closure
$\bB_R^{2n}$. The boundary $\dd \cX_i$ is a finite union 
of $d_i$ embedded $\bS^1$.

\medskip
Removing $N (\geq 1)$ open (resp. closed) disks from $S_g$, whose 
closures do not intersect, yields the compact (resp. open and bounded) 
connected surface $S_{g,N}$ with $N$ boundary components. 

\medskip
For $i = 1,2$, let 
$$
N_i^a := d_i + \sum_{j=1}^{s_i} c(\by_i^j) \;\; {\rm and} \;\;
\cC_i^a := \cX_i \cup \left( \cup_{j=1}^{s_i} \cC_i^j\right). 
$$
Therefore the compact (resp. open and bounded) surface with 
boundary $X_i^a\setminus \cC_i^a$ is smoothly diffeomorphic to 
the compact (resp. open and bounded) surface with boundary 
$S_{g_i,N_i^a}$.

\medskip\noindent
Obviously, (a) implies (b). We will show that (b) implies (c) implies (a).

\medskip\noindent
$\bullet$ Assume (b). Let $h:X_1^a \to X_2^a$ be a homeomorphism. 
Since $X_i^a$ is LNE, a point $\by_i$ of $X_i^a$ is singular if and
only if $c(\by_i) \geq 2$. Since $h$ maps each germ 
$$
(X_1^a\setminus \by,\by) \to (X_2^a\setminus h(\by),h(\by))
$$
we deduce that
$$
h(Y_1) = Y_2, \;\; {\rm and} \;\;
c(\by) = c(h(\by)) \;\; \forall \by \in Y_1.
$$
For $R$ large enough, the number of connected components $X_i \setminus 
B_R^{2n}$ are equal for $i=1,2$, that is $d_1 = d_2$. 
Thus we deduce that $g_1 = g_2$ by \eqref{eq:euler-affine-0}. 
Condition (c) is then satisfied.

\medskip\noindent
$\bullet$ Assume (c). Let $d:=d_1 = d_2$ and $g:=g_1 =g_2$, $s:=s_1=s_2$.
For each $i=1,2$ and each $j=1,\ldots,s$ we can assume without 
loss of generality that $c(\by_1^j) = c(\by_2^j)$.

We deduce $X_1^a\setminus \cC_1^a$
and $X_2^a \setminus \cC_2^a$ are both diffeomorphic to $S_{g,N^a}$
for $N^a := N_1^a = N_2^a$.

Since $X_1^a$ and $X_2^a$ are both LNE, we further deduce that 
$\cX_1$ and $\cX_2$ are bi-Lipschitz homeomorphic. So are $\cC_1^j$ and 
$\cC_2^j$ for each $j=1,\ldots,s$.

We can glue all these pieces together to $S_{g,N^a} = X_i \setminus \cC_i^a$ 
along the boundary $\dd \cC_i^a$ to produce a bi-Lipschitz 
homeomorphism $(X_1^a,d_{X_1^a}) \to (X_2^a,d_{X_2^a})$.
Thus (a) is verified.
\end{proof}

\bigskip
As a straightforward consequence of Theorem \ref{thm:classify} we obtain the 
full Lipschitz classification of LNE curves:
\begin{corollary}\label{prop:classify}
Let $X_1$ be a LNE complex algebraic curve of 
$\C^{n_1}$ and $X_2$ be a LNE complex algebraic curve of $\C^{n_2}$. Then 
$X_1$ and $X_2$ are bi-Lipschitz homeomorphic if and only if they are 
homeomorphic.
\end{corollary}

We obtain in our LNE setting the following
\begin{corollary}\label{cor:prol-LNE}
Let $X^a$ be a LNE curve of $\Cn$ with $n \geq 3$. There exists a Zariski 
dense open subset of linear projections $\Cn\to\C^2$ mapping $X^a$ onto 
a LNE curve of $\C^2$.
\end{corollary}
\begin{proof}
The linear projections $\Cn\to\C^2$, identified with their kernels,
 consist of the Grassmann manifold $\bG(n-2,n)$ of the $(n-2)$ planes of 
$\Cn$. Consider the Gauss mapping of $X^a$:
$$
\tau : X^a\setminus X_\sing \to \bG(1,n), \;\; \bx \to T_\bx X.
$$
The closure $G_0(X)$ of 
$\tau(X^a\setminus X)$ is either a curve of $\bG(1,n)$ or a 
finite set of points. Let 
$$
G(X) := G_0(X) \cup \left(\cup_{\lbd \in X^\infty} \C\lbd \right)
$$
Thus $G(X)$ is a sub-variety of $\bG(1,n)$ of dimension $0$ or $1$.
Consider the following sub-variety
$$
V := \{(\lbd,\lbd',P)\in G(X)\times G(X) \times \bG(n-2,n) : \dim_\C \C\lbd + 
\C\lbd' + P \leq n-1\}
$$
We check that $\dim_\C V = 2n-5$ and that $W$, the Zariski closure of the
projection of $V$ onto $\bG(n-2,n)$, has also dimension $2n-5$. 
Therefore
$$
\Omg(X) := \bG(n-2,n) \setminus W
$$
is Zariski open and dense. Let
$\pi_P: \Cn \to \C^2 = P^\perp$ be the orthogonal projection onto the 
orthogonal $P^\perp$ of $P \in \bG(n-2,n)$. By definition of $\Omg(X)$, we 
check that 
$$
X_P^a := \pi_P(X^a)
$$
has exactly $d = deg(X^a)$ points at infinity at each of which $X_P^a$
is transverse with the line at infinity. Therefore the degree of $X_P^a$ is 
$d$. Let $\bz$ be a point of $X_{P,\sing}^a$ and let 
$\pi_P^{-1}(\bz) = \{\bx_1,\ldots,\bx_q\}$. For each $i=1,\ldots,q$, the
number of distinct lines in the tangent cone of $X$ at $\bx_i$ is 
$c(\bx_i)$. 
The definition of $\Omg(X)$ implies that the tangent cone of $X_P^a$ at
$\bz$ will consist of $\sum_{i=1}^q c(\bx_i)$ distinct lines. Since $X$
is LNE and $P\in\Omg(X)$, the germ $(X_P^a,\bz)$ is a finite union of 
pairwise transverse non-singular curve germs. Thus
$X_P^a$ is locally LNE at $\bz$. Since $X^a$ is connected so is $X_P^a$.
Theorem \ref{thm:main} shows that $X_P^a$ is LNE. 
\end{proof}

To end the section let us remark that the above classification allows more 
types than represented by plane algebraic curves.

\begin{corollary}\label{cor:LNE-not-plane}
There exist LNE curves of $\Cn$, with $n\geq 3$, which are not
Lipschitz equivalent to a plane curve. 
In particular,   
there are  LNE curves of $\Cn$ not Lipschitz equivalent 
to any of their plane projections.
\end{corollary}
\begin{proof}
Indeed, for each $n\geq 3$ and each positive 
integer $d$, there is a non-negative integer $p(d,n)$, polynomial in $d$, 
such that for every $g$ such that $0\leq g\leq p(d,n)$, there
is an irreducible non-singular (and non-degenerate) curve of $\bP^n$
with degree $d$ and genus $g$ (see \cite{GrPe,Cil}). Thus, there 
exist non-singular space curves with  $2g  \neq (d-1)(d-2)$. The general 
affine trace of such a projective curve is LNE by Theorem 
\ref{thm:main-proj}. By Theorem \ref{thm:classify}, such an affine trace 
cannot be 
Lipschitz equivalent to a non-singular plane curve, which must satisfy the 
equality $2g  = (d-1)(d-2)$. It can neither be equivalent to any of its 
plane projections since the previous genus argument implies that these must 
have singular points. The claim follows from Theorem \ref{thm:classify}.
\end{proof}

Note that generic affine traces of non-singular irreducible projective
curves with genus zero and degree $d \geq 3$ are the simplest examples of LNE
curves which are not Lipschitz equivalent to any plane curve, since $DSGM 
= (d,0,0)$ cannot be realized by a plane curve, which is opposite to the 
results of \cite{Tei,Fer} in the local case.
%
%
%
%
%
%
%
%
%
%
%
%
%
%
%
%
%
%
%
%
%
%
%
%
%
%
%
%
%
%
%
%
%
%
%
%
%
%
%
%
%
%
%
%
%
%
%
%
%
%
%
%
%
%
%
%
%
%
%
%
%
%

\end{document}